\documentclass[10pt]{article}

\usepackage[utf8]{inputenc}
\usepackage[T2A]{fontenc}
\usepackage[english]{babel}
\usepackage{amsthm, amsfonts, amsmath, amssymb}
\usepackage{enumitem}
\usepackage[title]{appendix}
\usepackage{graphicx, epstopdf}
\usepackage{hyperref}

\DeclareMathOperator{\im}{im}
\DeclareMathOperator{\Ind}{Ind}
\DeclareMathOperator{\Nul}{Nul}
\DeclareMathOperator{\supp}{supp}

\DeclareMathOperator{\tr}{tr}

\renewcommand{\C}{\mathbb C}
\renewcommand{\Re}{\mathop{\mathrm{Re}}\nolimits}

\newcommand{\B}{\mathbb B}
\newcommand{\E}{\mathbb E}
\newcommand{\R}{\mathbb R}
\newcommand{\Z}{\mathbb Z}

\newcommand{\TC}{\mathrm{TC}}

\AtEndDocument{%
\par
\medskip
\begin{tabular}{@{}l@{}}%
\textsc{Faculty of Mathematics, National Research University Higher School of Economics,}\\

\smallskip

\textit{E-mail address}: \texttt{vomedvedev@hse.ru}\\
\textsc{Institute for Information Transmission Problems (Kharkevich Institute),}\\
\textsc{Faculty of Mathematics, National Research University Higher School of Economics,}\\
\textsc{Independent University of Moscow,}\\
\textit{E-mail address}: \texttt{eamorozov\_1@edu.hse.ru}
\end{tabular}}

\title{On the Index of Fraser-Sargent-type minimal surfaces}
\author{Vladimir Medvedev \and Egor Morozov}
\date{}

\theoremstyle{definition}

\theoremstyle{plain}
\newtheorem{theorem}{Theorem}[section]

\newtheorem{corollary}[theorem]{Corollary}
\newtheorem{proposition}[theorem]{Proposition}
\newtheorem{claim}[theorem]{Claim}
\newtheorem{conjecture}[theorem]{Conjecture}

\theoremstyle{remark}
\newtheorem{remark}{Remark}

\begin{document}

\maketitle

\begin{abstract}
Fraser-Sargent surfaces are free boundary minimal surfaces in the four-dimensional unit Euclidean ball. Extended infinitely they define immersed minimal surfaces in the Euclidean space. In the present paper we compute the Morse index and the nullity of  these extended minimal surfaces. The parts of these surfaces outside the ball are exterior free boundary minimal surfaces. We provide a numerical evidence that they are stable. As a corollary of these results we obtain a lower bound on the index of Fraser-Sargent surfaces inside the ball. The obtained lower bound is not sharp. We provide computational experiments and state a conjecture about an improved index lower bound. Independently of it we also find an upper bound on the index of Fraser-Sargent surfaces inside the ball.
\end{abstract}

\section{Introduction}

\subsection{Overview and main results}\label{sec:intro-ov}

Let $\Sigma$ be a surface and $(M,g)$ a Riemannian manifold. An immersion $u\colon \Sigma \to (M,g)$ is called \textit{minimal} if it is a critical point of the area functional 
$$
Area[u]=\int_\Sigma dA.
$$
In this paper we always identify $\Sigma$ with its image $u(\Sigma)$ in $M$ and say that $\Sigma$ is a \textit{minimal surface} in $(M,g)$. Equivalently, one can say that the surface $\Sigma$ is minimal in $(M,g)$ if its mean curvature vanishes. The equivalence of these two definitions immediately follows from the Euler-Lagrange equation for the area functional. In the case when $\Sigma$ and $M$ have non-empty boundaries we say that $\Sigma$ is a \textit{free boundary minimal surface} (\textit{FBMS} for short) in $(M,g)$ if it is a critical point of the area functional among all variations leaving the boundary of $\Sigma$ on the boundary of $M$. We always assume that an FBMS is \emph{properly embedded}, i.e. $\partial\Sigma=\partial M\cap \Sigma$. If $\Sigma$ is an FBMS, then the minimality condition implies that it has zero mean curvature and it meets the boundary $\partial M$ orthogonally.

One of the central problems of the theory of minimal surfaces is their classification for a given ambient Riemannian manifold. In this paper we consider the classification by the \textit{Morse index}. Informally speaking, the Morse index (or simply the index) of a minimal surface is the maximal number of linearly independent infinitesimal variations decreasing the area of the surface up to the second order. If there are no such variations, then we say that our minimal surface is \textit{stable} in $(M,g)$. 

In this paper we mostly consider two ambient Riemannian manifolds: the $n-$dimensional Euclidean space $\E^n$ and $\E^n\setminus\mathring\B^n$, which is the complement to the open unit ball $\mathring\B^n$ in $\E^n$. The minimality condition in these spaces implies that $\Sigma$ cannot be compact. Following the terminology of the paper~\cite{mazet2022free} we call the FBMS in $\E^n\setminus\mathring\B^n$ the \textit{exterior free boundary minimal surfaces} (\textit{EFBMS} for short). We also consider some applications to the theory of FBMS in $\B^n$.

The study of minimal surfaces and their indices in $\mathbb E^n$ has a long history. Without claiming to be complete, we mention only some results related to our current work.

\medskip

1) \textit{Finite total Gauss curvature.} In this paper we consider only minimal surfaces in $\mathbb E^n$ of finite total Gauss curvature. The classical theorem of Fisher-Colbrie and Nayatani (\cite{fischer1985complete,nayatani1990morse}) states that the finiteness of the total Gauss curvature implies the finiteness of the index (in fact, in $\E^3$ both conditions are equivalent). By the classical theorem of Huber and Osserman (\cite{huber1958subharmonic,osserman2013survey}) a minimal surface of finite total Gauss curvature is conformally equivalent to a compact Riemannian surface with finite number of punctures called ends. The genus of a minimal surface of finite total Gauss curvature is defined as the genus of this Riemann surface.  

\medskip

2) \textit{General index estimates.} To the best of our knowledge the most general index estimate which works for any Riemannian manifold was obtained by Ejiri and Micallef in the paper~\cite{ejiri2008comparison} (see Theorem~1.1). Using this estimate the authors obtained in the same paper (see Theorem~1.2) that for an orientable minimal surface $\Sigma$ in $\E^n$ one has
$$
\Ind(\Sigma)\leqslant\frac{1}{\pi}\int_\Sigma(-K)\,dA+2\gamma-2,
$$
where $\Ind(\Sigma)$ denotes the index of $\Sigma$, $K$ is the Gauss curvature of $\Sigma$, and $\gamma$ is the genus of $\Sigma$. In this paper we use the notation $\TC(\Sigma):=\int_\Sigma K\,dA$ for the total Gauss curvature. Sometimes we call $\TC(\Sigma)$ simply the \textit{total curvature}. 

In $\E^3$ the previous estimate can be improved as
$$
 \Ind(\Sigma) \leqslant -\frac{1}{\pi}\TC(\Sigma)+2\gamma-3.
$$
Moreover, Chodosh and Maximo proved in~\cite{chodosh2018topology} that in this case there exists the following lower bound
$$
 \Ind(\Sigma) \geqslant \frac{1}{3}\left(2\gamma+2\sum_{i=1}^r(d_i+1)-5\right),
$$
where $r$ is the number of ends and $d_i$ is the multiplicity of the $i-$th end. In the same paper the authors obtained the lower bound on the index of a non-orientable minimal surface in $\E^3$:
$$
 \Ind(\Sigma) \geqslant \frac{1}{3}\left(\gamma+2\sum_{i=1}^r(d_i+1)-4\right).
$$

\medskip

3) \textit{Some known small indices.} It's quite obvious that the plane in $\E^n$ is stable. It is the only orientable stable minimal surface in $\E^3$ (\cite{fischer1980structure,do1979stable,pogorelov1981stability}). In $\E^n$ with $n>3$ there exist other orientable stable minimal surfaces. For example, any holomorphic curve in $\E^4$ is automatically stable. Without requiring orientability and regularity, the Henneberg surface is stable (\cite{henneberg1875ueber}). Ros proved in~\cite{ros2006one} that there are no immersed non-orientable stable minimal surfaces in $\E^3$. The catenoid and Enneper's surface have index one (see for example~\cite{fischer1985complete,cheng1988index,tuzhilin1992morse}). They are unique immersed minimal surfaces of index one in $\E^3$ (\cite{lopez1989complete}). There are no immersed non-orientable minimal surfaces of index one in $\E^3$ (\cite{chodosh2018topology}). Also, there are no immersed orientable minimal surfaces in $\E^3$ of index 2 (\cite{chodosh2018topology}). The Chen-Gackstatter surfaces and the Richmond surface have index 3 (\cite{montiel2006schrodinger,tuzhilin1992morse}). These surfaces are immersed. There are no embedded orientable minimal surfaces in $\E^3$ of index 3 (\cite{chodosh2018topology}). The Costa-Hoffman-Meeks surface of genus $\gamma$ has index $2\gamma+3$ (\cite{nayatani1992morse,morabito2009index}) and the Jorge-Meeks $(k+1)-$ catenoid has  has index $2k+1$ (see \cite[Corollary 15 and the discussion below]{montiel2006schrodinger}). Finally, the Bryant minimal surface (\cite{bryant1984duality,rosenberg1986some}) has index 4 (see for example \cite[Corollary 26]{montiel2006schrodinger}).

\medskip

In the present paper we are interested in the following family of immersed minimal surfaces in $\E^4$. Fix a pair of relatively prime integers $(k,l)$ such that $k>l>0$ and let $T_{k,l}$ be the unique positive solution of the equation $k\tanh(kt)=l\coth(lt)$.
Then the image of the map
\begin{equation}\label{eq:intro-param1}
u_{k,l}(t,\theta)=
\frac{1}{r_{k,l}}(
k\sinh lt\cos l\theta,
k\sinh lt\sin l\theta,
l\cosh kt\cos k\theta,
l\cosh kt\sin k\theta
),\quad
t\in\R,\;\theta\in\R/2\pi\Z,
\end{equation}
where $r_{k,l}=\sqrt{k^2\sinh^2(lT_{k,l})+l^2\cosh^2 (kT_{k,l})}$, is a minimal surface in $\E^4$.
We call this surface a \emph{Fraser-Sargent surface} and denote it by $FS_{k,l}$. Some particular cases of them were studied in the papers~\cite{de1986some,mira2006complete,alarcon2020new,alarcon2021minimal}.
If $k$ is even (and hence $l$ is odd), then the map $u$ is invariant under the involution $(t,\theta)\mapsto (-t,\theta+\pi)$.
Thus, topologically $FS_{k,l}$ are M\"obius bands for even $k$ and cylinders for odd $k$.
If $k$ is even, then let $\widetilde{FS}_{k,l}$ be the orientable two-sheeted cover of $FS_{k,l}$. If $k$ is odd, then we let $\widetilde{FS}_{k,l}=FS_{k,l}$.

Note that $FS_{k,l}$ meets $\partial\B^4$ orthogonally due to the choice of constants.
Consider the image of the restriction of the map $u$ in~\eqref{eq:intro-param1} on $[-T_{k,l},T_{k,l}]\times(\R/2\pi\Z)$. We call it an \emph{interior Fraser-Sargent surface} and denote it by $IFS_{k,l}$.
These surfaces were introduced in~\cite{fan2015extremal} and studied in more details by Fraser and Sargent in the paper~\cite{fraser2021existence}.
For example, they proved that Fraser-Sargent FBMS are the only $\mathbb S^1-$invariant FBMS in $\B^4$.
Topologically, $IFS_{k,l}$ are bounded M\"obius bands for even $k$ or annuli for odd $k$.
Finally, define $\widetilde{IFS}_{k,l}$ similarly to $\widetilde{FS}_{k,l}$.



There are other parametrizations of $FS_{k,l}$.
First one can replace a pair of integer parameters $k,l$ by a single rational parameter $p=k/l$ by the substitution $t\mapsto \frac{t}{l},\theta\mapsto\frac{\theta}{l}$.
Then $FS_{k,l}$ is the image of map
\begin{equation}\label{eq:intro-param2}
u_p(t,\theta)=
\frac{1}{r_p}(
p\sinh t\cos\theta,
p\sinh t\sin\theta,
\cosh pt\cos p\theta,
\cosh pt\sin p\theta
),\quad
t\in\R,\;\theta\in\R/2\pi l\Z,
\end{equation}
where $r_p=\sqrt{p^2\sinh^2 T_p+\cosh^2 (p T_p)}$ and $T_p=l T_{k,l}$ (so that $T_p$ is the unique positive solution of the equation $p\tanh(pt)=\coth t$).

Finally, introducing the new variable $\zeta=e^t(\cos\theta+i\sin\theta)$, one obtains another useful parametrization for $FS_{k,l}$
\begin{equation}\label{eq:intro-param3}
v_{k,l}(\zeta)=
\frac{r}{2}\Re\left(
k\left(\zeta^l-\frac{1}{\zeta^l}\right),
-ik\left(\zeta^l+\frac{1}{\zeta^l}\right),
l\left(\zeta^k+\frac{1}{\zeta^k}\right),
-il\left(\zeta^k-\frac{1}{\zeta^k}\right)
\right),\quad
\zeta\in\C^*,
\end{equation}
cf.~\cite[\S2.8.11]{alarcon2021minimal} for the case $k=2,l=1$.
This parametrization can be used to compute the total curvature of $FS_{k,l}$.
Indeed, the same argument as in~\cite[\S2.8.11]{alarcon2021minimal} shows that
\begin{gather*}
\TC(\widetilde{FS}_{k,l})=-4\pi k,\quad\text{and hence}\quad
\TC(FS_{k,l})=\begin{cases}
-4\pi k,&\text{$k$ is odd,}\\
-2\pi k,&\text{$k$ is even.}
\end{cases}
\end{gather*}

Thus the index of $FS_{k,l}$ is finite. Now we are ready to state our main result.

\begin{theorem}\label{th:fs-ind}
Let $(k,l)$ be a pair of relatively prime integers such that $k>l>0$. Then 
$$
\Ind(FS_{k,l})=\begin{cases}
2k-1,&\text{$k$ is odd,}\\
k-1,&\text{$k$ is even,}
\end{cases}\quad\text{and}\quad
\Nul(FS_{k,l})=4.
$$
\end{theorem}

Here $\Nul(FS_{k,l})$ denotes the nullity of $FS_{k,l}$. Roughly speaking, the nullity is the maximal number of linearly independent variations of $FS_{k,l}$ on which the second variation of the area functional vanishes. 

\medskip

The second part of the paper is dedicated to the study of EFBMS. To the best of our knowledge the study of EFBMS was initiated in the paper~\cite{mazet2022free} in dimension 3. In this paper the authors construct a one-parametric family of catenoidal EFBMS and compute their indices. Depending on the parameter the index is either zero or one. 

We study the following EFBMS. Consider the map $u$ in~\eqref{eq:intro-param1}. The image of $[T_{k,l},+\infty)$ under this map is an EFBMS in $\E^4\setminus\mathring\B^4$.
We call this EFBMS an \emph{exterior Fraser-Sargent surface} and denote it by $EFS_{k,l}$. It is not hard to see that the Fisher-Colbrie-Nayatani theorem holds  for EFBMS , i.e. the finiteness of the total curvature implies the finiteness of the index (see for example Theorem~\ref{th:efbms-ub} below). The total curvature of $EFS_{k,l}$ can be easily found:

$$
\TC(EFS_{k,l})=\frac{1}{2}\bigl(\TC(\widetilde{FS}_{k,l})-\TC(\widetilde{IFS}_{k,l})\bigr)=-2\pi k, 
$$
since $\TC(\widetilde{IFS}_{k,l})=0$ by the Gauss-Bonnet theorem. Therefore the index of $EFS_{k,l}$ is finite. Unfortunately, we were not able to compute the index of $EFS_{k,l}$ rigorously. However we state the following conjecture

\begin{conjecture}\label{con:efs-stab}
For $k,l$ as in Theorem~\ref{th:fs-ind} we have $\Ind(EFS_{k,l})=0$, i.e. all the surfaces $EFS_{k,l}$ are stable.
\end{conjecture}

We provide a computer assisted proof of Conjecture~\ref{con:efs-stab} (see the discussion below).

\medskip

In the conclusion of this section we consider some applications of the obtained results to the theory of FBMS in $\B^n$. Let us briefly review some basic facts about FBMS in $\B^n$.

\medskip

1) \textit{General index estimates.} A general index bound which works for any compact FBMS and any ambient Riemannian manifold was proved by Lima in the paper~\cite{lima2017bounds}. In the case of $\B^n$ it reads
$$
\Ind_E(\Sigma)\leqslant \Ind(\Sigma)\leqslant \Ind_E(\Sigma)+\dim\mathcal M(\Sigma),
$$
where $\Ind_E(\Sigma)$ is the energy index of $\Sigma$ and $\mathcal M(\Sigma)$ is the moduli space of conformal classes on $\Sigma$. Ambrosio, Carlotto, Sharp in~\cite{ambrozio2018index} and Sargent in~\cite{sargent2017index} proved  in the case of $\B^3$ the following lower bound for orientable $\Sigma$
\begin{equation}\label{eq:intro-steklov}
\Ind(\Sigma)\geqslant \left[\frac{2\gamma+b+1}{3}\right],
\end{equation}
where $\gamma$ is the genus of $\Sigma$ and $b$ is the number of boundary components of $\Sigma$. This estimate implies that the index of an FBMS in $\B^3$ can be arbitrarily large.

\medskip

2) \textit{Relation to Spectral Geometry.} It is not hard to see that the components of an FBMS $u\colon \Sigma \to \B^n$ satisfy the following \textit{spectral Steklov problem}
$$
\begin{cases}\Delta v=0&\text{ in }\Sigma,\\
\partial_\eta v=v &\text{ on }\partial\Sigma,
\end{cases}
$$
where $\Delta$ is the Laplace-Beltrami operator on $\Sigma$ and $\eta$ is the outward unit normal to $\partial\Sigma$. In other words, $v$ is a \textit{Steklov eigenfunction with eigenvalue one}. The Steklov spectrum for compact $\Sigma$ consists of a discrete collection of eigenvalues with finite multiplicities. Fraser and Schoen in the paper~\cite{fraser2011first} gave a beautiful variational characterization of FBMS in $\B^n$ as \textit{extremal metrics} for the functional ``the $k-$th normalized Steklov eigenvalue''. Karpukhin and M\'etras in the paper~\cite{karpukhin2021laplace} gave a definition of the \textit{spectral index} of an FBMS in $\B^n$ as the number of Steklov eigenvalues less than one. Later, the first named author proved in~\cite{medvedev2023index} that
\begin{gather}\label{eq:intro-inds}
\Ind(\Sigma)\leqslant n\Ind_S(\Sigma)+\dim\mathcal M(\Sigma),
\end{gather}
where $\Ind_S(\Sigma)$ is the spectral index of $\Sigma$.

\medskip

3) \textit{Known indices.} The list of known results is quite short. Obviously, there are no stable FBMS in $\B^n$. It is not hard to see that the index of an equatorial disk in $\B^n$ is equal to $n-2$. The index of the critical catenoid is 4 (\cite{devyver2019index,tran2016index,smith2019morse,medvedev2023index}). Recently, it has been shown in~\cite{medvedev2023index} that the index of the critical M\"obius band is 5.

\medskip

The computation of the index of interior Fraser-Sargent surfaces is a challenging problem. Theorem~\ref{th:fs-ind} and Conjecture~\ref{con:efs-stab}  enable us to give the following lower bound 

\begin{corollary}\label{cor:ifs-lb}
For $k,l$ as in Theorem~\ref{th:fs-ind} we have
$$
\Ind(IFS_{k,l})\geqslant \begin{cases}
2k-1,&\text{$k$ is odd,}\\
k-1,&\text{$k$ is even.}
\end{cases}
$$
\end{corollary}

In fact this corollary follows directly from the following proposition which could be of independent interest

\begin{proposition}\label{pr:ifs-split}
Let $\Sigma$ be a minimal surface in $\E^n$. Suppose that $\B^n$ splits it into two pieces $E\Sigma$ and $I\Sigma$ which are exterior FBMS and interior FBMS respectively. Then
$$
\Ind(I\Sigma)+\Ind(E\Sigma)\geqslant \Ind(\Sigma).
$$
\end{proposition}

Also we were able to estimate the index of $IFS_{k,l}$ from above. Namely, the following theorem holds

\begin{theorem}\label{th:ifs-ub}
For $k,l$ as in Theorem~\ref{th:fs-ind} we have
$$
\Ind(IFS_{k,l})\leqslant\begin{cases}
8(k+l)-7,&\text{$k$ is odd,}\\
4(k+l)-7,&\text{$k$ is even.}
\end{cases}
$$
\end{theorem}

The proof of this theorem makes use of inequality~\eqref{eq:intro-inds} and the explicit computation of the spectral index of $IFS_{k,l}$ (see \S\ref{sec:ifs}).

\subsection{Discussion and open problems} 

1) \textit{On the proofs of Theorem~\ref{th:fs-ind} and Conjecture~\ref{con:efs-stab}.} In the proof of Theorem~\ref{th:fs-ind} we first find an appropriate upper bound on the index, then we find a lower bound, and, finally, we show that these two bounds are equal. The method of finding of an upper bound is inspired by the proof of Theorem~4.5 in the paper~\cite{ejiri2008comparison}. The proof of the lower bound is based on the same type of computations which were made by the second named author in the paper~\cite{morozov2022index}. Despite the fact that the proof of Conjecture~\ref{con:efs-stab} is computer assisted, one can see from \S\ref{sec:efs} that this proof is quite plausible. In fact, in \S\S\ref{sec:efs-sep}--\ref{sec:efs-elem} we show that Conjecture~\ref{con:efs-stab} follows from the following inequality
\begin{equation*}
\frac{x}{\cosh x}+\frac{2b-x}{\cosh(2b-x)}<
\frac{2b}{3b-2T_\infty}
\end{equation*}
for $x\geqslant 0$ and $b\geqslant T_\infty$. Here $T_\infty$ is the unique positive solution of the equation $t\tanh t=1$. This is the only place in the proof of Conjecture~\ref{con:efs-stab} where we use numerical computations. 


\medskip

2) \textit{Index of the Alarc\'on-Forstneri\v{c}-L\'opez M\"obius band.} For $k=2$ and $l=1$ the Fraser-Sargent surface $FS_{2,1}$ is nothing but the \textit{Alarc\'on-Forstneri\v{c}-L\'opez M\"obius band}. This surface was a first example of an embedded non-orientable minimal surface in $\E^4$ (\cite{alarcon2020new}, see also \cite{de1986some}). Theorem~\ref{th:fs-ind} then implies that $\Ind(FS_{2,1})=1$. As we mentioned in~\S\ref{sec:intro-ov} it was shown by Chodosh and Maximo in~\cite[Theorem 1.8]{chodosh2018topology} that there are no complete immersed non-orientable minimal surfaces of index 1 in $\mathbb E^3$. Theorem~\ref{th:fs-ind} shows that complete immersed non-orientable minimal surfaces of index 1 exist in $\mathbb E^4$. We suppose that the Alarc\'on-Forstneri\v{c}-L\'opez M\"obius band plays the same role as the catenoid in $\mathbb E^3$. Therefore the following conjecture seems plausible 

\begin{conjecture}
The Alarc\'on-Forstneri\v{c}-L\'opez M\"obius band is the only complete embedded non-orientable surface of index one in $\mathbb E^4$.
\end{conjecture}

\medskip

3) \textit{Other examples of EFBMS.} To the best of our knowledge the only known examples of non totally geodesic EFBMS are the family of catenoidal hypersurfaces in $\E^n\setminus \mathring\B^n$ for any $n \geqslant 3$ obtained in \cite{mazet2022free} and the surfaces $EFS_{k,l}$ in $\E^4\setminus \mathring\B^4$. It would be interesting to find other examples of EFBMS in $\E^n\setminus \mathring\B^n$.

\medskip



4) \textit{On Corollary~\ref{cor:ifs-lb}.} The result that we get in Corollary~\ref{cor:ifs-lb} particularly shows that in $\B^4$ there are no restrictions on the index from above: one can always find an FBMS with arbitrarily large index. The lower bound in Corollary~\ref{cor:ifs-lb} is not sharp. Indeed, for $k=2, l=1$ we obtain the critical M\"obius band. It was proved in the paper~\cite{medvedev2023index} that its index is 5 while by Corollary~\ref{cor:ifs-lb} $\Ind(IFS_{2,1})\geqslant 2$. Actually, the numerical experiments show that the estimate in Corollary~\ref{cor:ifs-lb} can be improved. More precisely, one has

\begin{conjecture}\label{con:ifs-lblite}
Let $\widetilde{IFS}_{k,l}$ denote the orientable cover of ${IFS}_{k,l}$ with $k,l$ as in Theorem~\ref{th:fs-ind} then
$$
\Ind(\widetilde{IFS}_{k,l})\geqslant 6k+2l-1.
$$
\end{conjecture}

For a more general statement see Conjecture~\ref{con:ifs-lb}.

\medskip

5) \textit{Relation to General Relativity.} The interest to the study of EFBMS may be related to General Relativity. It is well-known that the \textit{$n-$dimensional Schwarzschild space $Sch^n$}, which is one of the most popular objects of study in GR, is conformally equivalent to $\E^n\setminus\mathring\B^n$. Therefore it is natural to ask can one perturb an EFBMS in order to get an FBMS in $Sch^n$? It was shown in the paper~\cite{carlotto2017non} that for $n>3$ there exist catenoidal free boundary minimal hypersurfaces obtained as a perturbation of Euclidean hyper-catenoids. For $n=3$ the authors proved in the same paper that no perturbation of a Euclidean catenoid can produce a minimal surface in $Sch^3$. In fact, Carlotto and Mondino proved these statements for a more general setting of the so-called \emph{asymptotically Schwarzschildean manifolds}. In the case of $Sch^n$ with $n>3$ Barbosa and Moya in the paper~\cite{barbosa2021proper} found explicit examples of such rotationally symmetric catenoidal free boundary minimal hypersurfaces. It would be interesting to understand can one perturb $EFS_{k,l}$ in order to get an FBMS in $Sch^4$. 

\medskip

6) \textit{Spectral Geometry of $EFS_{k,l}$.} As we have already mentioned, there is a beautiful connection between the theory of FBMS in $\B^n$ and the spectral geometry of the Steklov problem. It would be interesting to find, if possible, a similar connection with Spectral Geometry in the case of EFBMS. Clearly, the components of an EFBMS $u\colon \Sigma \to \E^n\setminus \mathring\B^n$ satisfy the same spectral Steklov problem~\eqref{eq:intro-steklov} as the components of any FBMS in $\B^n$. Hence the components of $u$ are Steklov eigenfunctions with eigenvalue one as in the case of FBMS in $\B^n$. However, this Steklov problem does not have necessarily discrete Steklov spectrum since $\Sigma$ is not compact. This fact makes certain difficulties in the correct definition of the spectral index which was successfully used in the paper~\cite{medvedev2023index} in the case of FBMS in $\B^n$.

\subsection{Plan of the paper}

The paper is organized in the following way. In \S\ref{sec:def} we introduce necessary definitions and notation that we use throughout the paper. In \S\ref{sec:jac} we provide some computational results and particularly we compute the Jacobi operator for Fraser-Sargent surfaces. In \S\ref{sec:fs} we compute the index and the nullity of Fraser-Sargent surfaces in $\E^4$ and prove Theorem~\ref{th:fs-ind}. In \S\ref{sec:efs} we provide a numerical evidence that exterior Fraser-Sargent surfaces are stable (see Conjecture~\ref{con:efs-stab}). In \S\ref{sec:ifs} we prove Theorem~\ref{th:ifs-ub}, Proposition~\ref{pr:ifs-split} and Corollary~\ref{cor:ifs-lb}. Here we also discuss Conjecture~\ref{con:ifs-lblite}. Finally, in Appendix~\ref{sec:efbms} we prove a general index upper bound for EFBMS and in Appendix~\ref{sec:nonor} we prove some analogs of the Ejiri-Micallef inequalities for non-orientable surfaces in $\E^n$ and in a general Riemannian manifold.

\subsection{Acknowledgments}

The authors are grateful to Mikhail Karpukhin, Konstantin Loginov and Antonio Alarc\'on for fruitful discussions. During the work on the paper the authors were partially supported by the Theoretical Physics and Mathematics Advancement Foundation ``BASIS''. 
The work of the second author is also supported in part by the M\"obius Contest Foundation for Young Scientists.
This research is a part of the second author’s master thesis at the Higher School of Economics under the supervision of Alexei Penskoi.

\section{Notation and definitions}\label{sec:def}

In this paper we use the following notation and definitions:

\begin{itemize}
\item{$\E^n$ is the Euclidean $n$-dimensional space, $\B^n,\mathring\B^n$ are respectively the closed and the open unit balls centered at the origin of $\E^n$;}

\item{$\Sigma$ is either
\begin{itemize}
\item{a minimal surface in $\E^n$ without boundary given by an immersion $u\colon\Sigma\to\E^n$, or}
\item{a free boundary minimal surface (FBMS) in $\B^n$ given by an immersion $u\colon\Sigma\to\B^n$, or}
\item{an exterior free boundary minimal surface (EFBMS) in $\E^n\setminus\mathring\B^n$ given by an immersion $u\colon\Sigma\to\E^n\setminus\mathring\B^n$;}
\end{itemize}
}

\item{$\langle-,-\rangle$ and $|\cdot|$ are the standard Euclidean scalar product and norm in $\R^n$ respectively; we use the same notation for the induced scalar product on $\Sigma$;}

\item{$dA$ and $dL$ denote the area element on $\Sigma$ and the length element on $\partial\Sigma$ respectively;}

\item{$\Gamma(E)$ is the set of all smooth sections of a (real or complex) vector bundle $E$;}

\item{$T\Sigma$ and $N\Sigma$ are the tangent bundle of $\Sigma$ and the normal bundle to $\Sigma$ respectively;}

\item{for any vector $v\in\R^n$ the vectors $v^\bot$ and $v^\top$ are the projections of $v$ onto $\Gamma(N\Sigma)$ and $\Gamma(T\Sigma)$ respectively;}

\item{if $\partial\Sigma\ne\varnothing$, then $\eta\in\Gamma(T\Sigma|_{\partial\Sigma})$ is the outward unit normal vector field to $\partial\Sigma$;}

\item{$\nabla^\bot$ and $\nabla^\top$ are the connections in $N\Sigma$ and $T\Sigma$ respectively; $\nabla$ is the covariant derivative in $\E^n$;}

\item{the Laplace-Beltrami operator on $\Sigma$ is
$$
\Delta f=\sum_{i=1}^2 (e_i(e_i f)-(\nabla_{e_i}^\top e_i)f),\quad
f\in C^\infty(\Sigma),
$$
where $e_1,e_2$ is a local orthonormal frame in $\Gamma(T\Sigma)$;
}

\item{the Laplacian in the normal bundle is given by
\begin{equation}\label{eq:def-laplace}
\Delta^\bot X=\sum_{i=1}^2 (\nabla_{e_i}^\bot \nabla_{e_i}^\bot X-\nabla_{\nabla_{e_i}^\top e_i}^\bot X),\quad
X\in\Gamma(N\Sigma),
\end{equation}
where $e_1,e_2$ is a local orthonormal frame in $\Gamma(T\Sigma)$;
}

\item{the second fundamental form of $\Sigma$ is given by
$$
B(X,Y)=(\nabla_X Y)^\bot,\quad
X,Y\in\Gamma(T\Sigma),
$$
in particular, $b_{ij}=B(e_i,e_j)$;
}

\item{the Simons operator $\mathcal B\colon\Gamma(N\Sigma)\to\Gamma(N\Sigma)$ is given by
$$
\mathcal B(X)=\sum_{i,j=1}^2 \langle b_{ij},X\rangle b_{ij},\quad
X\in\Gamma(N\Sigma);
$$
}

\item{the Jacobi stability operator $L\colon\Gamma(N\Sigma)\to\Gamma(N\Sigma)$ is given by
$$
L(X)=\Delta^\bot X+\mathcal B(X),\quad
X\in\Gamma(N\Sigma);
$$
}

\item{the energy functional of the immersion $u$ is $E[u]=\frac{1}{2}\int_\Sigma |du|^2 dA$;}

\item{the total (Gauss) curvature of $\Sigma$ is
$$
\TC(\Sigma)=\int_\Sigma K\,dA+\,\int_{\partial\Sigma} k_g dL,
$$
where $K$ is the Gauss curvature of $\Sigma$ and $k_g$ is the (signed) geodesic curvature of $\partial\Sigma$ in $\Sigma$;
}

\item{if $\partial\Sigma=\varnothing$, then the second variation of the area towards a normal vector field $X\in\Gamma(N\Sigma)$ is the following quadratic form
$$
\delta^2 A(X)=-\int_\Sigma\langle L(X),X\rangle\,dA=
\int_\Sigma (|\nabla^\bot X|^2-\langle\mathcal B(X),X\rangle)\,dA,
$$
and the second variation of energy towards a vector field $V\in\Gamma(\Sigma\times\R^n)$ is
$$
\delta^2 E(V)=\int_\Sigma |\nabla V|^2 dA.
$$}

\item{if $\Sigma$ is an (E)FBMS, then the forms $\delta^2 A$ and $\delta^2 E$ both have contributions from the boundary.
More precisely, a vector field $V\in\Gamma(\Sigma\times\R^n)$ is called \emph{admissible} (see~\cite[Definition~1]{lima2017bounds}) if $V(p)\perp u(p)$ for any point $p\in\partial\Sigma$.
Then the second variation of area is given by
$$
\delta^2 A(X)=-\int_\Sigma\langle L(X),X\rangle\,dA+\int_{\partial\Sigma}(\langle X,\nabla_\eta^\bot X\rangle\pm|X|^2)\,dL=
\int_\Sigma (|\nabla^\bot X|^2-\langle\mathcal B(X),X\rangle)\,dA\pm\int_{\partial\Sigma}|X|^2,\,dL
$$
where $X\in\Gamma(N\Sigma)$, and the second variation of energy is given by
$$
\delta^2 E(V)=\int_\Sigma |\nabla V|^2 dA\pm\int_{\partial\Sigma} |V|^2 dL,
$$
where $V$ is an admissible vector field.
Here the sign is `$-$' in the case of FBMS in $\B^n$ and `+' in the case of EFBMS in $\E^n\setminus\mathring\B^n$;}

\item{for a bounded domain $\Omega\subset\Sigma$ the (Morse) \emph{index} of $\Omega$ is defined as the maximal dimension of a vector space $S\subset\Gamma(N\Sigma)$ such that $\delta^2 A$ is negative definite on $S$ and $\supp X\subset\Omega$ for each $X\in S$.
It is well-known that if $\partial\Sigma=\varnothing$, then this quantity is equal to the number of negative eigenvalues of the Dirichlet problem
$$
\begin{cases}
LX=-\lambda X&\text{ on }\Omega,\\
X=0&\text{ on }\partial\Omega,
\end{cases}
$$
where $X\in\Gamma(N\Sigma|_\Omega)$. If $\Sigma$ is an (E)FBMS, then $\Ind(\Omega)$ is equal to the number of negative eigenvalues of the mixed Dirichlet-Robin problem
$$
\begin{cases}
LX=-\lambda X&\text{ on }\Omega,\\
\frac{\partial X}{\partial\eta}\pm X=0&\text{ on }\partial\Omega\cap\partial\B^n,\\
X=0&\text{ on }\partial\Omega\setminus\partial\B^n,
\end{cases}
$$
where $X\in\Gamma(N\Sigma|_\Omega)$. Again, the sign is `$-$' in the case of FBMS and `+' in the case of EFBMS;}

\item{the (Morse) \emph{index} of $\Sigma$ is $\Ind(\Sigma)=\sup\Ind(\Omega)$, where the supremum is taken over all bounded domains $\Omega\subset\Sigma$. The surface $\Sigma$ is called \emph{stable} if $\Ind(\Sigma)=0$;}

\item{the \textit{energy index} $\Ind_E\Sigma$ is defined similarly to $\Ind(\Sigma)$ with $\delta^2 E$ instead of $\delta^2 A$ and the set of all admissible vector fields (or just $\Gamma(\Sigma\times\R^n)$ if $\partial\Sigma=\varnothing$) instead of $\Gamma(N\Sigma)$.
Note, however, that if $\partial\Sigma=\varnothing$ or $\Sigma$ is an EFBMS, then $\Ind_E\Sigma=0$ since the quadratic form $\delta^2 E$ is non-negative in both cases;}

\item{a vector field $X\in\Gamma(N\Sigma)$ is called a \emph{Jacobi field} if $LX=0$.
If $\partial\Sigma=\varnothing$, then the dimension of the space of all bounded Jacobi fields on $\Sigma$ is called the \emph{nullity} of $\Sigma$ and is denoted by $\Nul(\Sigma)$. The following claim is well-known.

\begin{claim}\label{cl:def-jacfields}
The projection of a Killing vector field in $\E^n$ on $N\Sigma$ is a Jacobi field on $\Sigma$.
In particular, the projection of any constant vector field in $\E^4$ on $N\Sigma$ is a bounded Jacobi field and $\Nul(\Sigma)\geqslant 4$ if $\Sigma$ is not contained in $\E^3$.
\end{claim}}


\end{itemize}




\section{Jacobi operator for Fraser-Sargent surfaces}\label{sec:jac}

In this section we compute the Jacobi stability operator $L$ for Fraser-Sargent surfaces.
The computations are very similar to~\cite[\S3.1]{morozov2022index}.
Since the computations are local, they are valid for all the surfaces $FS_{k,l},EFS_{k,l},IFS_{k,l}$ and their orientable covers.

Fix a pair of relatively prime integers $(k,l)$ such that $k>l>0$.
We put $u_{k,l}=u$ and $r_{k,l}=r$ for simplicity.
Let us introduce the function
$$
\rho(t)=\frac{kl}{r}\sqrt{\sinh^2 lt+\cosh^2 kt}=\frac{kl}{r}\sqrt{\cosh^2 lt+\sinh^2 kt}.
$$

\begin{proposition}\label{pr:jac-basis}
For any point $x\in FS_{k,l}$ the vectors
\begin{equation}\label{eq:jac-basis}
\begin{aligned}
e_1&=
\frac{u_\theta}{\rho(t)}=
\frac{kl}{r\rho(t)}\begin{pmatrix}
-\sinh lt\sin l\theta\\
\sinh lt\cos l\theta\\
-\cosh kt\sin k\theta\\
\cosh kt\cos k\theta
\end{pmatrix},
&n_1&=\frac{kl}{r\rho(t)}\begin{pmatrix}
-\sinh kt\cos l\theta\\
-\sinh kt\sin l\theta\\
\cosh lt\cos k\theta\\
\cosh lt\sin k\theta
\end{pmatrix},\\
e_2&=
\frac{u_t}{\rho(t)}=
\frac{kl}{r\rho(t)}\begin{pmatrix}
\cosh lt\cos l\theta\\
\cosh lt\sin l\theta\\
\sinh kt\cos k\theta\\
\sinh kt\sin k\theta
\end{pmatrix},
&n_2&=\frac{kl}{r\rho(t)}\begin{pmatrix}
-\cosh kt\sin l\theta\\
\cosh kt\cos l\theta\\
\sinh lt\sin k\theta\\
-\sinh lt\cos k\theta
\end{pmatrix}
\end{aligned}
\end{equation}
form an orthonormal basis of $T_x\E^4$.
Moreover, $e_1,e_2$ is a basis of $T_x FS_{k,l}$, $n_1,n_2$ is a basis of $N_x FS_{k,l}$,
and $z:=t+i\theta$ is a conformal coordinate on $FS_{k,l}$ with conformal factor $\rho(t)^2$.
\end{proposition}

The proof of this proposition is a direct verification.
In the sequel, all computations are made w.r.t. the basis~\eqref{eq:jac-basis}. The (local) section $f_1 n_1+f_2 n_2$ of $N(FS_{k,l})$ is denoted by $\left[\begin{smallmatrix} f_1\\ f_2 \end{smallmatrix}\right]$.

\begin{proposition}\label{pr:jac-simons}
The matrix of the Simons operator $\mathcal B\colon\Gamma(N(FS_{k,l}))\to\Gamma(N(FS_{k,l}))$ in the basis $n_1,n_2$ is given by
$$
\mathcal B=
\frac{2}{\rho(t)^2}\begin{bmatrix}
a(t)^2 & 0\\
0 & c(t)^2
\end{bmatrix},
$$
where
$$
a(t)=\frac{k^2 l^2}{r^2\rho(t)^2}(k\cosh lt\cosh kt-l\sinh lt\sinh kt),\quad
c(t)=\frac{k^2 l^2}{r^2\rho(t)^2}(l\cosh lt\cosh kt-k\sinh lt\sinh kt).
$$
\end{proposition}

\begin{proof}
We have
$$
b_{11}=\frac{u_{\theta\theta}}{\rho(t)^2}=-\frac{kl}{r\rho(t)^2}\begin{pmatrix}
l\sinh lt\cos l\theta\\
l\sinh lt\sin l\theta\\
k\cosh kt\cos k\theta\\
k\cosh kt\sin k\theta
\end{pmatrix},\quad
b_{12}=\frac{u_{t\theta}}{\rho(t)^2}=\frac{kl}{r\rho(t)^2}\begin{pmatrix}
l\cosh lt\sin l\theta\\
-l\cosh lt\cos l\theta\\
k\sinh kt\sin k\theta\\
-k\sinh kt\cos k\theta
\end{pmatrix},
$$
and $b_{22}=-b_{11}$ by minimality.
The entry $\mathcal B_{\alpha\beta}$ of the matrix $\mathcal B$ equals $\sum\limits_{i,j=1}^2\langle b_{ij},n_\alpha\rangle\langle b_{ij},n_\beta\rangle$ and the result follows by a direct computation.
\end{proof}

\begin{proposition}\label{pr:jac-deltabot}
The operator $\Delta^\bot\colon\Gamma(N(FS_{k,l}))\to\Gamma(N(FS_{k,l}))$ in the basis $n_1,n_2$ is given by
$$
\Delta^\bot\begin{bmatrix} f_1 \\ f_2 \end{bmatrix}=
\frac{1}{\rho(t)^2}\begin{bmatrix}
\Delta_0 f_1+2b(t)\rho(t)\partial_\theta f_2-b(t)^2 f_1\\
\Delta_0 f_2-2b(t)\rho(t)\partial_\theta f_1-b(t)^2 f_2
\end{bmatrix},
$$
where
$$
b(t)=\frac{k^2 l^2}{r^2\rho(t)^2}(l\sinh kt\cosh kt+k\sinh lt\cosh lt)
$$
and $\Delta_0=\partial_t^2+\partial_\theta^2$ is the flat Laplacian.
\end{proposition}

\begin{proof}
It follows easily from~\eqref{eq:def-laplace} that for any $n\in\Gamma(N(FS_{k,l}))$ and $f\in C^\infty(FS_{k,l})$ we have
\begin{equation}\label{eq:jac-laprod}
\Delta^\bot(fn)=f\Delta^\bot n+(\Delta f)n+2\sum_{i=1}^2 (e_i f)\nabla_{e_i}^\bot n.
\end{equation}
Therefore it suffices to calculate $\Delta^\bot n_\alpha$ and $\nabla_{e_j}^\bot n_\alpha$ for $i,\alpha=1,2$. We have
$$
\nabla_{e_1}n_1=\frac{kl}{r\rho(t)^2}\begin{pmatrix}
l\sinh kt\sin l\theta\\
-l\sinh kt\cos l\theta\\
-k\cosh lt\sin k\theta\\
k\cosh lt\cos k\theta
\end{pmatrix},\quad
\nabla_{e_2}n_1=-\frac{\rho'(t)}{\rho(t)^2}n_1+\frac{kl}{r\rho(t)^2}\begin{pmatrix}
-k\cosh kt\cos l\theta\\
-k\cosh kt\sin l\theta\\
l\sinh lt\cos k\theta\\
l\sinh lt\sin k\theta
\end{pmatrix}.
$$
From this it is easy to see that
$$
\langle\nabla_{e_1}^\bot n_1,n_2\rangle=-\langle\nabla_{e_1}^\bot n_2,n_1\rangle=-\frac{b(t)}{\rho(t)}\quad\text{and}\quad
\langle\nabla_{e_i}^\bot n_\alpha,n_\beta\rangle=0\text{ for all other $i,\alpha,\beta=1,2$.}
$$
Hence,
\begin{gather}
\nabla_{e_1}^\bot n_1=-\frac{b(t)}{\rho(t)}n_2,\quad
\nabla_{e_1}^\bot n_2=\frac{b(t)}{\rho(t)}n_1,\quad\label{eq:jac-e1n}\\
\nabla_{e_2}^\bot n_1=\nabla_{e_2}^\bot n_2=0.\label{eq:jac-e2n}
\end{gather}
Further, we have
$$
\langle\nabla_{e_1}^\top e_1,e_1\rangle=\frac{1}{2}e_1\langle e_1,e_1\rangle=0,\quad
\langle\nabla_{e_2}^\top e_2,e_1\rangle=\frac{1}{\rho(t)^3}\langle u_{tt},u_\theta\rangle=0,
$$
which together with~\eqref{eq:jac-e2n} gives
\begin{equation}\label{eq:jac-eenzero}
\nabla_{\nabla_{e_i}^\top e_i}^\bot n_\alpha=0\quad\forall i,\alpha=1,2.
\end{equation}
Combining~\eqref{eq:def-laplace},~\eqref{eq:jac-e1n}--\eqref{eq:jac-eenzero}, we get
\begin{equation}\label{eq:jac-deltan}
\Delta^\bot n_1=\nabla_{e_1}^\bot\nabla_{e_1}^\bot n_1=-\frac{b(t)^2}{\rho(t)^2}n_1,\quad
\Delta^\bot n_2=\nabla_{e_1}^\bot\nabla_{e_1}^\bot n_1=-\frac{b(t)^2}{\rho(t)^2}n_2.
\end{equation}
The desired result follows from~\eqref{eq:jac-laprod},~\eqref{eq:jac-e1n},~\eqref{eq:jac-deltan} and the relation $\rho(t)^2\Delta=\Delta_0$.
\end{proof}

\begin{proposition}\label{pr:jac-jac}
The Jacobi stability operator $L\colon\Gamma(N(FS_{k,l}))\to\Gamma(N(FS_{k,l}))$ in the basis $n_1,n_2$ has the form
$$
L\begin{bmatrix} f_1 \\ f_2 \end{bmatrix}=
\frac{1}{\rho(t)^2}\begin{bmatrix}
\Delta_0 f_1+2b(t)\rho(t)\partial_\theta f_2+(2a(t)^2-b(t)^2) f_1\\
\Delta_0 f_2-2b(t)\rho(t)\partial_\theta f_1+(2c(t)^2-b(t)^2) f_2
\end{bmatrix}.
$$
\end{proposition}

\begin{proof}
It immediately follows from Propositions~\ref{pr:jac-simons} and~\ref{pr:jac-deltabot}.
\end{proof}

\begin{remark}\label{rem:jac-params}
All computations in this section are made w.r.t. the parametrization~\eqref{eq:intro-param1}.
However, it is clear that the same computations work for the parametrization~\eqref{eq:intro-param2} as well: just substitute $k=p$ and $l=1$ in all formulas and use $T_p,r_p$ instead of $T_{k,l},r_{k,l}$ respectively.
By abuse of notation we use the same notation for both parametrizations.
Thus, for example, the function $\rho(t)$ is either $\frac{kl}{r_{k,l}}\sqrt{\sinh^2 lt+\cosh^2 kt}$ or $\frac{p}{r_p}\sqrt{\sinh^2 t+\cosh^2 pt}$ depending on the context.
We hope this will not cause any confusion since the parametrizations~\eqref{eq:intro-param1} and~\eqref{eq:intro-param2} are used in different sections.
\end{remark}

\section{Index of $FS_{k,l}$}\label{sec:fs}

Throughout this section we use parametrization~\eqref{eq:intro-param1}.

\subsection{The upper bound}\label{sec:fs-ub}

In this section we prove the following

\begin{proposition}\label{pr:fs-ub}
We have $\Ind(\widetilde{FS}_{k,l})+\Nul(\widetilde{FS}_{k,l})\leqslant 2k+3$ and $\Ind(\widetilde{FS}_{k,l})\leqslant 2k-1$.
\end{proposition}

The proof is based on the ideas of the paper~\cite{ejiri2008comparison}. However, applying~\cite[Theorem~3.2]{ejiri2008comparison} directly, we get
\begin{gather}\label{ineq:not}
\Ind(\widetilde{FS}_{k,l})+\Nul(\widetilde{FS}_{k,l})\leqslant
-\frac{1}{\pi}\TC(\widetilde{FS}_{k,l})+2=4k+2,
\end{gather}
which is not satisfactory. One can hope that the analog of the Ejiri-Micallef inequality for non-orintable surfaces in $\E^n$ would be satisfactory in the case of $FS_{k,l}$ (for those of them who are non-orientable). In Appendix~\ref{sec:nonor} we obtain this analog which in the case of $FS_{k,l}$ yields:
$$
\Ind({FS}_{k,l})+\Nul({FS}_{k,l})\leqslant
-\frac{1}{\pi}\TC(FS_{k,l})+3=2k+3.
$$
However, this is not satisfactory neither. 

Our idea is to improve the upper bound~\eqref{ineq:not} using the symmetry of $\widetilde{FS}_{k,l}$, as it is done in~\cite[Theorem~4.5]{ejiri2008comparison}.

First of all, let us recall the classical Hubert-Osserman compactification construction.
Let $\Sigma$ be a minimal surface in $\E^n$ without boundary and with finite total curvature.
Then $\Sigma$ is conformally equivalent to a compact Riemann surface $\Sigma_c$ with finitely many punctures, which correspond to the ends of $\Sigma$.
The subbundle $T\Sigma$ of $\Sigma\times\R^n$ extends to a subbundle $\tau$ of $\Sigma_c\times\R^n$
(actually, the extension is given by the pullback of the tautological bundle over the Grassmanian $\mathrm{Gr}(n,2)$ by the extended Gauss map).
Let $\nu$ be the normal bundle to $\tau$ in $\Sigma_c\times\R^n$; then $\nu|_\Sigma=N\Sigma$ and $\tau\oplus\nu=\Sigma_c\times\R^n$.
Moreover, the quadratic forms $\delta^2 A$ and $\delta^2 E$ extend to $\Gamma(\nu)$ and $\Gamma(\Sigma_c\times\R^n)$ respectively; thus, the indices $\Ind(\Sigma_c)$ and $\Ind_E(\Sigma_c)$ are well-defined.
One can show that (see~\cite[Remark~2]{nayatani1990morse} and also~\cite[Corollary~2]{fischer1985complete})
\begin{equation}\label{eq:fs-sigmacind}
\Ind(\Sigma_c)=\Ind(\Sigma)\quad\text{and}\quad
\Ind_E(\Sigma_c)=\Ind_E(\Sigma)=0.
\end{equation}
For the complexification $\tau_\C$ of $\tau$ we have the decomposition $\tau_\C=\tau^{1,0}\oplus\tau^{0,1}$.
Note that for an antiholomorphic vector bundle $E$ over $\Sigma$ or $\Sigma_c$ the Hermitian metric gives an isomorphism $E\cong \bar E^*$.
Thus we can consider $E$ as a holomorphic vector bundle as well.
In particular, we can consider $T^{0,1}\Sigma$ and $\tau^{0,1}$ as holomorphic vector bundles.
Denote $\mathcal L=\tau^{0,1}\otimes\Lambda^{1,0}\Sigma_c$ and let $H^0(\mathcal L)$ be the set of holomorphic sections of $\mathcal L$.
Also, let us denote $\nabla^{1,0}=\nabla_z\otimes dz$.

Now we specialize to the case where $\Sigma=\widetilde{FS}_{k,l}$.
We have by Proposition~\ref{pr:jac-basis} that $z=t+\theta i$ is a conformal coordinate on $\Sigma$.
Put $\omega_0:=\frac{1}{\rho^2}u_{\bar z}\otimes dz\in\Gamma(T^{0,1}\Sigma\otimes\Lambda^{1,0}\Sigma)$.

\begin{proposition}\label{pr:fs-nablan}
The section $\omega_0$ extends to a holomorphic section of $\mathcal L$. Moreover, we have
$$
(\nabla^{1,0}n_1)^\top=-\rho(t)a(t)\omega_0,\quad
(\nabla^{1,0}n_2)^\top=i\rho(t)c(t)\omega_0.
$$
In particular, $(\nabla^{1,0}n_1,\omega_0)$ is purely real and $(\nabla^{1,0}n_2,\omega_0)$ is purely imaginary, where $(-,-)$ stands for the induced Hermitian product in $\mathcal L$.
\end{proposition}

\begin{proof}
The section $\omega_0$ maps to a holomorphic section $\frac{1}{2}dz\otimes dz$ under the isomorphism $T^{0,1}\Sigma\otimes\Lambda^{1,0}\Sigma\cong\Lambda^{1,0}\Sigma\otimes\Lambda^{1,0}\Sigma$.
Since $|\frac{1}{2}dz\otimes dz|=\frac{1}{\rho(t)^2}\to 0$, this section extends by zero to a holomorphic section of $(\tau^{1,0})^*\otimes\Lambda^{1,0}\Sigma_c\cong\mathcal L$, which proves the first part of the proposition.
The second part is a direct computation. We have
\begin{multline*}
(\nabla^{1,0}n_1)^\top=\left(\langle\nabla_z n_1,u_z\rangle\frac{u_{\bar z}}{|u_z|^2}+\langle\nabla_z n_1,u_{\bar z}\rangle\frac{u_z}{|u_{\bar z}|^2}\right)\otimes dz=
-\langle n_1,u_{zz}\rangle\frac{2}{\rho^2}u_{\bar z}\otimes dz=\\
-\langle n_1,u_{tt}-iu_{t\theta}\rangle\frac{1}{\rho^2}u_{\bar z}\otimes dz=
-\rho(t)a(t)\omega_0,
\end{multline*}
where we have used that $u_{z\bar z}=0$ by minimality. A similar computation for $(\nabla^{1,0}n_2)^\top$ completes the proof.
\end{proof}

Define
$$
\mathcal M_{\mathcal L}=\{g\in\mathcal M\colon (g)+(\omega_0)\geqslant 0\},
$$
where $\mathcal M$ is the space of meromorphic functions on $\Sigma_c$ and $(\cdot)$ stands for a divisor. Also define
$$
\sigma\colon\Sigma\to\Sigma,\quad
\sigma(t,\theta)=(-t,\theta),\quad\text{that is,}\quad
\sigma(z)=-\bar z.
$$
Obviously, $\sigma$ is an isometry. Observe that $\sigma$ extends to an antiholomorphic involution of $\Sigma_c$ ($\sigma$ permutes the ends of $\Sigma$). Moreover, the map
$$
C^\infty(\Sigma_c)\otimes_\R\C\to C^\infty(\Sigma_c)\otimes_\R\C,\quad
g\mapsto\sigma^*\bar g
$$
restricts to a linear involution $\mathcal M_{\mathcal L}\to\mathcal M_{\mathcal L}$.
In particular, there exists a basis $\{g_1,\ldots,g_\mu\}$ of $\mathcal M_{\mathcal L}$ consisting of the eigenfunctions of this involution (so that $\sigma^*\bar g_\alpha=\pm g_\alpha$).
Finally, define
\begin{equation}\label{eq:fs-iota}
\iota\colon\Gamma(N\Sigma)\to\Gamma(N\Sigma),\quad
\iota\begin{bmatrix} f_1\\ f_2 \end{bmatrix}=
\begin{bmatrix} -\sigma^*f_1\\ \sigma^*f_2\end{bmatrix}.
\end{equation}
Then $\iota$ extends to a linear involution $\Gamma(\nu)\to\Gamma(\nu)$.
Let $\Gamma_+(\nu),\Gamma_-(\nu)$ be the $+1$ and $-1$ eigenspaces of $\Gamma(\nu)$ w.r.t. this involution.
For a section $X\in\Gamma(\nu)$ denote by $X_\pm=\frac{1}{2}(X\pm\iota X)$ the projection of $X$ onto $\Gamma_\pm(\nu)$.

\begin{proposition}\label{pr:fs-d2iota}
We have $\delta^2 A(X)=\delta^2 A(X_+)+\delta^2 A(X_-)$ for any $X\in\Gamma(\nu)$.
\end{proposition}

\begin{proof}
Clearly, it suffices to show that $\delta^2 A(\iota X)=\delta^2 A(X)$.
It follows from~\eqref{eq:fs-iota} that
$$
\langle \iota X,\iota Y\rangle=\sigma^*\langle X,Y\rangle\quad\forall X,Y\in\Gamma(\nu).
$$
Also, it is easy to see from Proposition~\ref{pr:jac-jac} that $\iota$ commutes with $L$ (this comes from the fact that the functions $a(t)$ and $c(t)$ are even and $b(t)$ is odd). Hence, for any $X\in\Gamma(\nu)$ with compact support we have
\begin{multline*}
\delta^2 A(\iota X)=-\int_\Sigma \langle L(\iota X),\iota X \rangle\,dA=
-\int_\Sigma \langle\iota L(X),\iota X \rangle\,dA=
-\int_\Sigma \sigma^*\Big(\langle L(X),X \rangle\Big)\,dA=\\
-\int_\Sigma \langle L(X),X \rangle\,dA=\delta^2 A(X),
\end{multline*}
where we used that $\sigma$ is an isometry.
\end{proof}

\begin{proof}[Proof of Proposition~\ref{pr:fs-ub}]
Let $S$ (respectively, $S_+,S_-$) be the maximal subspace of $\Gamma(\nu)$ (respectively, $\Gamma_+(\nu),\Gamma_-(\nu)$) on which $\delta^2 A\leqslant 0$.

Let us show that $S=S_+\oplus S_-$. It is clear from Proposition~\ref{pr:fs-d2iota} that $S\supset S_+\oplus S_-$.
Now suppose that there exists $X\in S$ such that $X\notin S_+\oplus S_-$. Then $X_+\notin S_+$ or $X_-\notin S_-$.
In particular, there exists $Y\in S_+\oplus S_-$ such that $\delta^2 A(X_++Y_+)\geqslant 0$ and $\delta^2 A(X_-+Y_-)\geqslant 0$ and at least one of the inequalities is strict. Since $X+Y\in S$, we obtain
$$
0\geqslant\delta^2 A(X+Y)=\delta^2 A(X_++Y_+)+\delta^2 A(X_-+Y_-)>0
$$ 
again by Proposition~\ref{pr:fs-d2iota}. This contradiction shows that $S=S_+\oplus S_-$.

Since $S=S_+\oplus S_-$, there exists a basis
$\{X^1,\ldots,X^q\}$ of $S$, consisting of the eigenfunctions of $\iota$ (so that $\iota X^j=\pm X^j$).
Let $X^j=f_1^j n_1+f_2^j n_2$. By Proposition~\ref{pr:fs-nablan} for each $X^j$ we have
$$
(\nabla^{1,0}X_j)^\top=
f_1^j (\nabla^{1,0}n_1)^\top+f_2^j (\nabla^{1,0}n_2)^\top=
\rho(t)(-a(t)f_1^j+c(t)f_2^j\cdot i)\omega_0
$$
and for each $X^j$ and $g_\alpha$ we have
\begin{multline*}
((\nabla^{1,0}X^j)^\top,g_\alpha\omega_0)=
\int_\Sigma \rho(t)(-a(t)f_1^j+c(t)f_2^j\cdot i)\bar g_\alpha |\omega_0|^2 dA=\\
\int_\Sigma \sigma^*\Big(\rho(t)(-a(t)f_1^j+c(t)f_2^j\cdot i)\bar g_\alpha |\omega_0|^2 dA\Big)=
\int_\Sigma \sigma^*\rho(t)(-\sigma^*a(t)\sigma^*f_1^j+\sigma^*c(t)\sigma^*f_2^j\cdot i)\sigma^*\bar g_\alpha \sigma^*(|\omega_0|^2 dA)=\\
\pm\int_\Sigma \rho(t)(-a(t)f_1^j-c(t)f_2^j\cdot i)g_\alpha |\omega_0|^2 dA=
\pm\overline{((\nabla^{1,0}X^j)^\top,g_\alpha\omega_0)}.
\end{multline*}
Here we used that $\sigma$ is an isometry, that the functions $\rho(t),a(t),c(t)$ are even (and therefore $\sigma^*-$invariant), and also our choice of bases in $S$ and $\mathcal M_{\mathcal L}$.
It follows that $((\nabla^{1,0}X^j)^\top,g_\alpha\omega_0)$ is either real or pure imaginary for each $X^j$ and $g_\alpha$.
Now in the same way as in~\cite[end of the proof of Theorem~4.5]{ejiri2008comparison} we obtain
$$
\Ind(\Sigma)+\Nul(\Sigma)\leqslant -\frac{1}{2\pi}\TC(\Sigma)+3=
2k+3,
$$
and the second inequality of Proposition~\ref{pr:fs-ub} follows from Claim~\ref{cl:def-jacfields}.
\end{proof}

\subsection{The lower bound}

In this section we prove the following

\begin{proposition}\label{pr:fs-lb}
$\Ind(\widetilde{FS}_{k,l})\geqslant 2k-1$.
\end{proposition}

Take $T>0$ and consider the domain $\Omega_T:=u((-T,T)\times(\R/2\pi\Z))\subset FS_{k,l}$.
Again, let $\widetilde\Omega_T$ be the orientable two-sheeted cover of $\Omega_T$ if $k$ is even and put $\widetilde\Omega_T=\Omega_T$ if $k$ is odd.
Recall from~\S\ref{sec:def} that one can define $\Ind\widetilde\Omega_T$ as the number of negative eigenvalues of the following problem
\begin{equation}\label{eq:fs-eigen}
\begin{cases}
LX=-\lambda X&\text{ on }\widetilde\Omega_T,\\
X=0&\text{ on }\partial\widetilde\Omega_T,
\end{cases}
\end{equation}
where $X\in\Gamma(N(\widetilde{FS}_{k,l})|_{\widetilde\Omega_T})$.

\begin{proposition}\label{pr:fs-sep}
Fix $T>0$. For each $m=0,1,2,\ldots$ consider the following matrix Sturm-Liouville weighted eigenvalue problem with Dirichlet boundary conditions
\begin{equation}\label{eq:fs-sl}
\begin{cases}
-h''(t)+A(m,t)h(t)=\lambda\rho(t)^2 h(t),\\
h(-T)=h(T)=0,
\end{cases}
\end{equation}
where $h(t)=\left(\begin{smallmatrix}h_1(t)\\ h_2(t)\end{smallmatrix}\right)$ is a vector function and
\begin{equation}\label{eq:fs-pot}
A(m,t)=\begin{pmatrix}
b(t)^2+m^2-2a(t)^2 & -2mb(t)\\
-2mb(t) & b(t)^2+m^2-2c(t)^2
\end{pmatrix}.
\end{equation}
Then

1) the eigenspace of the problem~\eqref{eq:fs-eigen} with eigenvalue $\lambda$ has a basis consisting of eigensections of the form
\begin{equation}\label{eq:fs-eigensec}
\begin{bmatrix}
h_1(t)\cos m\theta\\
h_2(t)\sin m\theta
\end{bmatrix}\quad\text{and}\quad
\begin{bmatrix}
-h_1(t)\sin m\theta\\
h_2(t)\cos m\theta
\end{bmatrix},
\end{equation}
where $h(t)=\left(\begin{smallmatrix}h_1(t)\\ h_2(t)\end{smallmatrix}\right)$ is a solution of~\eqref{eq:fs-sl} for some $m=0,1,2,\ldots$;

2) one can choose the vector functions $h(t)$ from~1) in such a way that one of the functions $h_1(t),h_2(t)$ is even in $t$ and the other one is odd in $t$;

3) if $h_1(t),h_2(t)$ are chosen as in~2), $k$ is even, and $m\ne 0$, then both eigensections~\eqref{eq:fs-eigensec} descend to $N(FS_{k,l})|_{\Omega_T}$ as soon as the number $m$ and the function $h_1(t)$ are of the same parity.
Otherwise, if the number $m\ne 0$ and the function $h_1(t)$ are of different parity, none of the eigensections~\eqref{eq:fs-eigensec} descend to $N(FS_{k,l})|_{\Omega_T}$.
\end{proposition}

\begin{proof}
1) Since the operator $L$ commutes with $\partial_\theta$, we see that the eigenspace of the problem~\eqref{eq:fs-eigen} has a basis consisting of eigensections of the form~\eqref{eq:fs-eigensec} for some functions $h_1(t),h_2(t)$.
Then a direct computation involving Proposition~\ref{pr:jac-jac} shows that $h(t)$ solves~\eqref{eq:fs-sl}.

2) This follows from the observation that the matrix Sturm-Liouville differential operator $-\frac{d^2}{dt^2}+A(m,t)$, corresponding to the problem~\eqref{eq:fs-sl}, is invariant under the map
$\left(\begin{smallmatrix}h_1(t)\\ h_2(t)\end{smallmatrix}\right)\mapsto\left(\begin{smallmatrix}-h_1(-t)\\ h_2(-t)\end{smallmatrix}\right)$ and hence there exists a joint basis of vector eigenfunctions for~\eqref{eq:fs-sl}.

3) Let $k$ be even. Then it follows from~\eqref{eq:jac-basis} that the sections $n_1$ and $n_2$ are respectively even and odd w.r.t. the involution $(t,\theta)\mapsto (-t,\theta+\pi)$.
Hence, a section $\left[\begin{smallmatrix}f_1(t)\\ f_2(t)\end{smallmatrix}\right]$ descends to $N(FS_{k,l})|_{\Omega_T}$ if and only if $f_1$ and $f_2$ are respectively even and odd w.r.t. the same involution. The claim now follows.
\end{proof}

Let
$$
\lambda_1^+(m,T)\leqslant\lambda_2^+(m,T)\leqslant\ldots\leqslant\lambda_i^+(m,T)\leqslant\ldots
$$
be the eigenvalues of the problem~\eqref{eq:fs-sl} such that the first component of the corresponding vector eigenfunction is even and the second one is odd.
Consider the quadratic form
$$
Q_{m,T}[h]=\int_{-T}^T (|h'(t)|^2+\langle A(m,t)h(t),h(t)\rangle)\,dt=
\langle h'(t),h(t)\rangle\Bigr|_{-T}^T+\int_{-T}^T\langle -h''(t)+A(t,m)h(t),h(t)\rangle\,dt,
$$
and let $Q_{m,\infty}[h]=\lim\limits_{T\to\infty} Q_{m,T}[h]$.
We use the following well-known variational characterization of $\lambda_1^+(m,T)$.

\begin{claim}\label{cl:fs-var}
The eigenvalue $\lambda_1^+(m,T)$ of the problem~\eqref{eq:fs-sl} has the following characterization
$$
\lambda_1^+(m,T)=\inf_{h\in H_{0+}^1([-T,T],\R^2)\setminus\{0\}}\frac{Q_{m,T}[h]}{\int_{-T}^T \rho(t)^2|h(t)|^2 dt},
$$
where
$$
H_{0+}^1([-T,T],\R^2)=\{h\in H_0^1([-T,T],\R^2)\colon h_1\text{ is even, }h_2\text{ is odd}\}
$$
In particular, $\lambda_1^+(m,T)<0$ if and only if $Q_{m,t}[h]<0$ for some $h\in H_{0+}^1([-T,T],\R^2)$.

(Here $H_0^1([-T,T],\R^2)$ denotes the space of Sobolev vector functions on $[-T,T]$ vanishing at $t=\pm T$ in the sense of trace.)
\end{claim}

Consider the constant vector fields $\partial_3=(0,0,1,0)$ and $\partial_4=(0,0,0,1)$ in $\E^4$. We have
\begin{equation}\label{eq:fs-jacfields}
(\partial_3)^\bot=
\frac{kl}{r\rho(t)}\begin{bmatrix}
\cosh lt\cos k\theta\\
\sinh lt\sin k\theta
\end{bmatrix},\quad
(\partial_4)^\bot=
\frac{kl}{r\rho(t)}\begin{bmatrix}
\cosh lt\sin k\theta\\
-\sinh lt\cos k\theta
\end{bmatrix}.
\end{equation}
Put
$$
\hat h_1(t)=\frac{kl}{r\rho(t)}\cosh lt,\quad
\hat h_2(t)=\frac{kl}{r\rho(t)}\sinh lt,\quad
\hat h(t)=\begin{pmatrix}\hat h_1(t)\\ \hat h_2(t)\end{pmatrix}.
$$

\begin{proposition}\label{pr:fs-qneg}
We have $Q_{k,\infty}[\hat h]=0$ and $Q_{m,\infty}[\hat h]<0$ for each $m\in [0,k-1]$.
\end{proposition}

\begin{proof}
It follows from Claim~\ref{cl:def-jacfields} that the normal vector fields~\eqref{eq:fs-jacfields} are Jacobi fields on $FS_{k,l}$.
This implies that $-\hat h''(t)+A(m,t)\hat h(t)=0$ and
$$
Q_{k,\infty}[\hat h]=\lim_{T\to\infty}Q_{k,T}[\hat h]=\lim_{T\to\infty}\left(\langle\hat h'(t),\hat h(t)\rangle\Bigr|_{-T}^T\right)=0,
$$
since $\hat h(t),\hat h'(t)\to 0$ as $t\to\infty$. This proves the first claim. To prove the second one we observe that
$$
Q_{m,\infty}[\hat h]=Q_{0,\infty}[\hat h]+\int_{-\infty}^{+\infty} B(m,t)dt,
$$
where
$$
B(m,t)=m^2(\hat h_1(t)^2+\hat h_2(t)^2)-4mb(t)\hat h_1(t)\hat h_2(t).
$$
Note that $B(m,t)$ is just a quadratic polynomial in $m$ with roots
$$
m_1=0,\quad\text\quad m_2=\frac{4b(t)\hat h_1(t)\hat h_2(t)}{\hat h_1(t)^2+\hat h_2(t)^2}=b(t)\tanh 2lt.
$$
It is easy to see that $0<m_2<k$ for any fixed $t\in \R$ and thus $B(m,t)<B(k,t)$ for $m\in [0,k-1]$.
This implies that
$$
Q_{m,\infty}[\hat h]<Q_{k,\infty}[\hat h]=0,
$$
as desired.
\end{proof}

\begin{proposition}\label{pr:fs-lamneg}
For each $m\in [0,k-1]$ there exists $T>0$ such that $\lambda_1^+(m,T)<0$.
\end{proposition}

\begin{proof}
Fix some $m\in [0,k-1]$. Let $\eta\colon\R\to [0,1]$ be a smooth even function such that
$$
\eta|_{[-1,1]}\equiv 1,\quad
\eta|_{\R\setminus (-2,2)}\equiv 0,
$$
and put $\eta_T(t)=\eta\left(\frac{2t}{T}\right)$. Then we have
$$
Q_{m,T}[\eta_T\hat h]=Q_{m,T/2}[\hat h]+2R(T),\quad\text{where}\quad
R(T)=\int_{T/2}^T (|(\eta_T(t)\hat h(t))'|^2+\eta_T^2(t)\langle A(m,t)\hat h(t),\hat h(t)\rangle)\,dt.
$$
Since $\hat h,\hat h',A(m,t)\hat h\in L^2(\R,\R^2)$, it is clear that $R(T)\to 0$ as $T\to\infty$. Hence,
$$
Q_{m,T}[\eta_T\hat h]\to Q_{m,\infty}[\hat h]\quad\text{as}\quad T\to\infty.
$$
By Proposition~\ref{pr:fs-qneg} we have $Q_{m,\infty}[\hat h]<0$.
In particular, $Q_{m,T}[\eta_T\hat h]<0$ for any sufficiently large $T$.
Since $\eta_T\hat h\in H_{0+}^1([-T,T],\R^2)$, we obtain
$\lambda_1^+(m,T)<0$ by Claim~\ref{cl:fs-var}.
\end{proof}

\begin{proof}[Proof of Proposition~\ref{pr:fs-lb}]
By Proposition~\ref{pr:fs-lamneg} there exists $T>0$ such that $\lambda_1^+(m,T)<0$ for each $m\in [0,k-1]$.
If $m>0$, then each vector eigenfunction of~\eqref{eq:fs-sl} with eigenvalue $\lambda_1^+(m,T)$ gives two eigensections~\eqref{eq:fs-eigensec} and thus contributes~2 to $\Ind\widetilde\Omega_T$.
If $m=0$, then it may happen that one of the eigensections~\eqref{eq:fs-eigensec} vanishes identically (this happens whenever either $h_1(t)\equiv 0$ or $h_2(t)\equiv 0$).
However, since $h(t)\not\equiv 0$ both eigensections~\eqref{eq:fs-eigensec} cannot vanish identically simultaneously.
Thus the eigenvalue $\lambda_1^+(0,T)$ contributes at least~1 to $\Ind\widetilde\Omega_T$ and we obtain $\Ind(\widetilde{FS}_{k,l})\geqslant\Ind\widetilde\Omega_T\geqslant 2k-1$.
\end{proof}

\begin{proof}[Proof of Theorem~\ref{th:fs-ind}]
For odd $k$ the result follows immediately from Propositions~\ref{pr:fs-ub},~\ref{pr:fs-lb}, and Claim~\ref{cl:def-jacfields} since $\widetilde{FS}_{k,l}=FS_{k,l}$ in this case.
Let $k$ be even. Then still $\Ind(\widetilde{FS}_{k,l})=2k-1$ and $\Nul(\widetilde{FS}_{k,l})=4$.
This means that for any sufficiently large $T$ we have $\lambda_1^+(m,T)<0$ for $m\in [0,k-1]$, and these are all negative eigenvalues of~\eqref{eq:fs-sl}. In particular,
\begin{equation}\label{eq:fs-lambda12}
\lambda_1(0,T)<0\leqslant\lambda_2(0,T).
\end{equation}
By part~3) of Proposition~\ref{pr:fs-sep} the eigenvalue $\lambda_1^+(m,T)<0$ contributes~2 to $\Ind(\Omega_T)$ for each even $m\in [2,k-2]$ and does not contribute to $\Ind(\Omega_T)$ for odd $m$.

Now consider $m=0$. It follows from~\eqref{eq:fs-lambda12} that the eigenvalue $\lambda_1(0,T)$ has multiplicity~1.
This implies that $h_1(t)\equiv 0$ or $h_2(t)\equiv 0$ for the corresponding vector eigenfunction $h(t)$ (otherwise $\left(\begin{smallmatrix}h_1(t)\\ 0\end{smallmatrix}\right)$ and $\left(\begin{smallmatrix}0\\ h_2(t)\end{smallmatrix}\right)$ would be two linearly independent vector eigenfunctions).
In particular, for $m=0$ exactly one of the eigensections~\eqref{eq:fs-eigensec} does not vanish identically and $\lambda_1^+(0,T)$ contributes~1 to $\Ind(\Omega_T)$. We get
$$
\Ind(FS_{k,l})=\Ind(\Omega_T)=1+2\cdot\frac{k-2}{2}=k-1.
$$
Finally, observe that
$$
4\leqslant\Nul(FS_{k,l})\leqslant\Nul(\widetilde{FS}_{k,l})=4.
$$
Here the first inequality is contained in Claim~\ref{cl:def-jacfields} and the second one follows from the fact that any bounded Jacobi field on $FS_{k,l}$ can be lifted to a bounded Jacobi field on $\widetilde{FS}_{k,l}$.
Hence, $\Nul(FS_{k,l})=4$, which concludes the proof.
\end{proof}

\section{Stability of $EFS_{k,l}$}\label{sec:efs}

Throughout this section we use parametrization~\eqref{eq:intro-param2} with the domain restricted to $[T_p,+\infty)\times(\R/2\pi l\Z)$ (recall Remark~\ref{rem:jac-params} concerning the notation).

\subsection{Separation of variables}\label{sec:efs-sep}

Take $T>T_p$ and consider the domain $\Omega_T:=u((T_p,T)\times(\R/2\pi l\Z))\subset EFS_{k,l}$.
Our goal is to show that $\Ind(\Omega_T)=0$ for each $T>T_p$.
Recall from~\S\ref{sec:def} that one can define $\Ind(\Omega_T)$ as the number of negative eigenvalues of the following problem
\begin{equation}\label{eq:efs-eigen}
\begin{cases}
LX=-\lambda X&\text{ on }\Omega_T,\\
\partial_\eta X+X=0&\text{ on }\partial\Omega_T\cap\partial\B^n,\\
X=0&\text{ on }\partial\Omega_T\setminus\partial\B^n,
\end{cases}
\end{equation}
where $X\in\Gamma(N(EFS_{k,l})|_{\Omega_T})$.

\begin{proposition}\label{pr:efs-sep}
Fix $T>T_p$. For each $m=0,1,2,\ldots$ consider the following matrix Sturm-Liouville weighted eigenvalue problem with mixed Robin-Dirichlet boundary conditions
\begin{equation}\label{eq:efs-sl}
\begin{cases}
-h''(t)+A(m,t)h(t)=\lambda\rho(t)^2 h(t),\\
h'(T_p)=\rho(T_p)h(T_p),\quad
h(T)=0,
\end{cases}
\end{equation}
where $h(t)=\left(\begin{smallmatrix}h_1(t)\\ h_2(t)\end{smallmatrix}\right)$ is a vector function and $A(m,t)$ as in~\eqref{eq:fs-pot}.
Then the eigenspace of the problem~\eqref{eq:efs-eigen} with eigenvalue $\lambda$ has a basis consisting of eigensections of the form
\begin{equation}\label{eq:efs-eigenfun}
\begin{bmatrix}
h_1(t)\cos m\theta\\
h_2(t)\sin m\theta
\end{bmatrix}\quad\text{and}\quad
\begin{bmatrix}
-h_1(t)\sin m\theta\\
h_2(t)\cos m\theta
\end{bmatrix},
\end{equation}
where $h(t)=\left(\begin{smallmatrix}h_1(t)\\ h_2(t)\end{smallmatrix}\right)$ is a solution of~\eqref{eq:efs-sl} for some $m=0,1,2,\ldots$
\end{proposition}

The proof is almost the same as the proof of Proposition~\ref{pr:fs-sep}, part~1).
One should only check additionally that the boundary condition on $\partial\Omega_T\cap\partial\B^n$ in~\eqref{eq:efs-eigen} leads to the Robin boundary condition in~\eqref{eq:efs-sl}.
This follows immediately from the fact that $\eta=-\frac{u_t}{\rho(t)}$ along the boundary.

\subsection{A scalar Sturm-Liouville problem}

It follows from Proposition~\ref{pr:efs-sep} that in order to prove the stability of $EFS_{k,l}$ it suffices to show that the eigenvalues of the problem~\eqref{eq:efs-sl} are positive for any $T>T_p$.
To this end we consider the quadratic form
$$
Q_{m,T}[h]=\int_{T_p}^T (|h(t)|^2+\langle A(m,t)h(t),h(t)\rangle)\,dt+\rho(T_p)|h(T_p)|^2,
$$
where $h(t)=\left(\begin{smallmatrix}h_1(t)\\ h_2(t)\end{smallmatrix}\right)$ is a vector function on $[T_p,T]$.
Similarly to Claim~\ref{cl:fs-var} we have

\begin{claim}\label{cl:efs-var}
The problem~\eqref{eq:efs-sl} has only positive eigenvalues if and only if the form $Q_{m,T}$ is positive definite on the space $H_{0,r}^1([T_p,T],\R^2)$ of Sobolev vector functions vanishing at $T$ in the  sense of trace.
\end{claim}

Observe that
$$
Q_{m,T}[h]=Q_T^{(1)}[h_1]+Q_T^{(2)}[h_2]+\widetilde Q_{m,T}[h],
$$
where
\begin{gather*}
Q_T^{(1)}[h_1]=\int_{T_p}^T (h_1'(t)^2-2a(t)^2 h_1(t)^2)\,dt+\rho(T_p)h_1(T_p)^2,\quad
Q_T^{(2)}[h_2]=\int_{T_p}^T (h_2'(t)^2-2c(t)^2 h_2(t)^2)\,dt+\rho(T_p)h_2(T_p)^2,\\
\widetilde Q_{m,T}[h]=\int_{T_p}^T \bigl((b(t)h_1(t)-mh_2(t))^2+(b(t)h_2(t)-mh_1(t))^2\bigr)\,dt.
\end{gather*}
Here $Q_T^{(1)}$ and $Q_T^{(2)}$ are defined on scalar functions and $\widetilde Q_{m,T}$ is defined on vector functions.
Note that $\widetilde Q_{m,T}>0$.
Also, since $a(t)>c(t)$, we obtain $Q_T^{(2)}>Q_T^{(1)}$.
It follows that if $Q_T^{(1)}$ is positive definite on $H_{0,r}^1[T_p,T]$, then $Q_{m,T}$ is positive definite on $H_{0,r}^1([T_p,T],\R^2)$.

Now we need the following result from scalar Sturm-Liouville theory.

\begin{proposition}\label{pr:efs-oscil}
Consider a scalar Sturm-Liouville problem on $[t_0,+\infty)$
\begin{equation}\label{eq:efs-gensl}
\begin{cases}
g''(t)+q(t)g(t)=0,\\
g'(t_0)=bg(t_0),
\end{cases}
\end{equation}
where $q\colon [t_0,+\infty)\to\R$ is continuous and $b>0$. Then the following conditions are equivalent
\begin{enumerate}[label=(\roman*)]
\item{for each $T>t_0$ the quadratic form
$$
Q_T[h]=\int_{t_0}^T \bigl(h'(t)^2-q(t)h(t)\bigr)\,dt+bh(t_0)^2
$$
is positive definite on $H_{0,r}^1[t_0,T]$;
}
\item{the solution of the problem~\eqref{eq:efs-gensl} does not vanish on $[t_0,+\infty)$.}
\end{enumerate}
In particular, if
$$
\begin{cases}
g''(t)+q_1(t)g(t)=0,\\
g'(t_0)=b_1 g(t_0)
\end{cases}\quad\text{and}\quad
\begin{cases}
g''(t)+q_2(t)g(t)=0,\\
g'(t_0)=b_2 g(t_0)
\end{cases}
$$
are two scalar Sturm-Liouville problems on $[t_0,+\infty)$ such that $q_1(t)>q_2(t)$ and $b_1<b_2$ and the solution of the first problem does not vanish on $[t_0,+\infty)$, then the solution of the second problem does not vanish on $[t_0,+\infty)$ as well.
\end{proposition}

\begin{proof}
Actually, both (i) and (ii) are equivalent to
\begin{enumerate}[label=(\roman*), start=3]
\item{\textit{for each $T>t_0$ all eigenvalues of the problem 
$$
\begin{cases}
g''(t)+q(t)g(t)=-\lambda g(t),\\
g'(t_0)=bg(t_0),\quad g(T)=0
\end{cases}
$$
on $[t_0,T]$ are positive. 
}}
\end{enumerate}
Indeed, (i) $\Leftrightarrow$ (iii) by variational characterization of eigenvalues and (ii) $\Leftrightarrow$ (iii) by the Sturm Oscillation Theorem.
\end{proof}

Using Proposition~\ref{pr:efs-oscil}, we see that the positive definiteness of $Q_T^{(1)}$ on $H_{0,r}^1[T_p,T]$ is equivalent to the following

\begin{conjecture}\label{con:efs-scal}
The solution of the following Sturm-Liouville problem
\begin{equation}\label{eq:efs-scal}
\begin{cases}
g''(t)+2a(t)^2 g(t)=0,\\
g'(T_p)=\rho(T_p)g(T_p),
\end{cases}
\end{equation}
does not vanish on $[T_p,+\infty)$.
\end{conjecture}

Conjecture~\ref{con:efs-stab} follows from Conjecture~\ref{con:efs-scal} by the observations above.

\subsection{An auxiliary scalar problem}\label{sec:efs-auxscal}

The problem~\eqref{eq:efs-scal} is equivalent to the following problem on $[1,+\infty)$
\begin{equation}\label{eq:efs-nscal}
\begin{cases}
\tilde g''(t)+T_p^2\cdot 2a(T_p t)^2\tilde g(t)=0,\\
\tilde g'(1)=T_p\rho(T_p)\tilde g(1),
\end{cases}
\end{equation}
which is easily seen from the substitution $\tilde g(t)=g(T_p t)$.
In the sequel, we work with the problem~\eqref{eq:efs-nscal}.

Consider an auxiliary problem
\begin{equation}\label{eq:efs-aux}
\begin{cases}
\tilde g''(t)+\frac{2}{9}\left(t-\frac{2}{3}\right)^{-2}\tilde g(t)=0,\\
\tilde g'(1)=\tilde g(1).
\end{cases}
\end{equation}
Note that the function $\left(t-\frac{2}{3}\right)^\frac{1}{3}$ solves this problem and does not vanish on $[1,+\infty)$.

\begin{conjecture}\label{con:efs-aux}
We have $T_p a(T_p t)<\frac{1}{3}\left(t-\frac{2}{3}\right)^{-1}$.
\end{conjecture}

Let us explain how Conjecture~\ref{con:efs-scal} (and hence Conjecture~\ref{con:efs-stab}) follows from Conjecture~\ref{con:efs-aux}.
First note that $T_p\rho(T_p)>1$.
Indeed, the function $t\coth t$ increases on $(0,+\infty)$, and hence
$$
T_p\rho(T_p)=T_p\coth T_p>\lim_{t\to 0}t\coth t=1.
$$
Since the solution of the auxiliary problem~\eqref{eq:efs-aux} does not vanish on $[1,+\infty)$, Conjecture~\ref{con:efs-scal} follows from the second part of Proposition~\ref{pr:efs-oscil} applied to~\eqref{eq:efs-nscal} and~\eqref{eq:efs-aux}.

\subsection{An elementary inequality}\label{sec:efs-elem}

Further we reduce Conjecture~\ref{con:efs-aux} to elementary inequality~\eqref{eq:efs-fineq}.
First we need the following

\begin{proposition}\label{pr:efs-ptp}
Let $T_\infty$ be the unique positive solution of the equation $t\tanh t=1$. Then
\begin{enumerate}[label=(\arabic*)]
\item{$pT_p>T_\infty$;}
\item{$(p-1)T_p<T_\infty$;}
\item{the function $\frac{a}{\cosh at}$ is a decreasing function in $a$ for $a>T_\infty$ and $t\geqslant 1$.}
\end{enumerate}
\end{proposition}

\begin{proof}
(1) The function $t\coth t$ increases on $(0,+\infty)$. Hence,
$$
pT_p\tanh pT_p=T_p\coth T_p>\lim_{t\to 0} t\coth t=1.
$$
Since the function $t\tanh t$ is increasing on $(0,+\infty)$ and $T_\infty\tanh T_\infty=1$, we obtain that $pT_p>T_\infty$.

(2) The function $t(\coth t-p\tanh t)$ decreases on $(0,+\infty)$ for any $p\geqslant 1$. Hence,
\begin{multline*}
(p-1)T_p\tanh (p-1)T_p=
(p-1)T_p\cdot\frac{\tanh pT_p-\tanh T_p}{1-\tanh pT_p\tanh T_p}=\\
(p-1)T_p\cdot\frac{\coth T_p-p\tanh T_p}{p-1}=
T_p(\coth T_p-p\tanh T_p)<
\lim_{t\to 0} t(\coth t-p\tanh t)=1,
\end{multline*}
where in the second equality we used that $p\tanh pT_p=\coth T_p$.
Since the function $t\tanh t$ is increasing on $(0,+\infty)$ and $T_\infty\tanh T_\infty=1$, we obtain that $(p-1)T_p<T_\infty$.

(3) We have
$$
\frac{d}{da}\frac{a}{\cosh at}=\frac{\cosh at-at\sinh at}{\cosh^2 at}=\frac{1-at\tanh at}{\cosh at}.
$$
For $a>T_\infty$ and $t\geqslant 1$ we have $at>T_\infty$ and $1-at\tanh at<0$, i.e. the function $\frac{a}{\cosh at}$ is a decreasing function in $a$.
\end{proof}

\smallskip

Let us return to Conjecture~\ref{con:efs-aux}. Observe that
$$
a(t)=\frac{p\cosh t\cosh pt-\sinh t\sinh pt}{\sinh^2 t+\cosh^2 pt}=
\frac{1}{2}\left(\frac{p-1}{\cosh (p-1)t}+\frac{p+1}{\cosh (p+1)t}\right).
$$
So Conjecture~\ref{con:efs-aux} is equivalent to the inequality
\begin{equation}\label{eq:efs-prefineq}
\frac{(p-1)T_p}{\cosh (p-1)T_p t}+\frac{(p+1)T_p}{\cosh (p+1)T_p t}<\frac{2}{3t-2}.
\end{equation}
By Proposition~\ref{pr:efs-ptp} we have
$$
(p+1)T_p>2T_\infty-(p-1)T_p>T_\infty\quad\text{and}\quad
\frac{(p+1)T_p}{\cosh (p+1)T_p t}<\frac{(2T_\infty-(p-1)T_p)}{\cosh (2T_\infty-(p-1)T_p)t}.
$$
Thus the inequality~\eqref{eq:efs-prefineq} follows from the inequality
$$
\frac{(p-1)T_p}{\cosh (p-1)T_p t}+\frac{(2T_\infty-(p-1)T_p)}{\cosh (2T_\infty-(p-1)T_p)t}<
\frac{2}{3t-2}.
$$
Substituting $x=(p-1)T_p t$ and $b=T_\infty t$, we arrive at the following inequality
\begin{equation}\label{eq:efs-fineq}
\frac{x}{\cosh x}+\frac{2b-x}{\cosh(2b-x)}<
\frac{2b}{3b-2T_\infty}.
\end{equation}
Conjecture~\ref{con:efs-stab} follows from~\eqref{eq:efs-fineq} for $b\geqslant T_\infty$ and $x\geqslant 0$.

\subsection{Numerical proof of the elementary inequality}

Unfortunately, we are not able to prove~\eqref{eq:efs-fineq}.
Instead, we verify it using a huge amount of numerical computations.

First note that the left hand side of~\eqref{eq:efs-fineq} is invariant under the map $x\mapsto 2b-x$.
Hence it suffices to verify~\eqref{eq:efs-fineq} for $x\in (-\infty,b]$.
Further, put $f(x)=\frac{x}{\cosh x}$, then $f'(x)=\frac{1-x\tanh x}{\cosh x}$.
Hence the function $f(x)$ increases on $[0,T_\infty]$, decreases on $[T_\infty,+\infty)$, and attains its maximum at $x=T_\infty$.
In particular, if $x<T_\infty$, then $f(x)<f(T_\infty)$ and $f(2b-x)<f(2b-T_\infty)$.
Hence it suffices to verify~\eqref{eq:efs-fineq} for $x\in [T_\infty,b]$. Consider two cases.

\textit{Case 1:} $x\in[T_\infty,b-\frac{1}{2}]$. Numerically, we have
$$
T_\infty=1.1996\ldots\quad\text{and}\quad
f(x)\leqslant\frac{T_\infty}{\cosh T_\infty}=0.662\ldots<\frac{2}{3}.
$$
Since the function $f(2b-x)$ increases on $[T_\infty,b-\frac{1}{2}]$ we have
$$
f(x)+f(2b-x)<\frac{2}{3}+f(b+\frac{1}{2})=\frac{2}{3}+\frac{b+\frac{1}{2}}{\cosh(b+\frac{1}{2})}.
$$
Thus it suffices to show that
$$
\frac{2}{3}+\frac{b+\frac{1}{2}}{\cosh(b+\frac{1}{2})}<\frac{2b}{3b-2T_\infty}\quad\text{for}\quad b\geqslant T_\infty,
$$
which after the substitution $c=b+\frac{1}{2}$ and algebraic manipulations is equivalent to
\begin{equation}\label{eq:efs-num1}
\cosh c>\frac{9}{4T_\infty}c^2-\left(\frac{9}{8T_\infty}+\frac{3}{2}\right)c\quad\text{for}\quad c\geqslant T_\infty-\frac{1}{2}.
\end{equation}
Clearly, the inequality~\eqref{eq:efs-num1} holds for any sufficiently large $c$ (say for $c^2>\frac{54}{T}$ we have $\cosh c>\frac{c^4}{24}>\frac{9}{4T_\infty}c^2$).
On a finite interval one can check this inequality numerically (see Fig~\ref{fig:efs-num1}).
\begin{figure}[h!]
\begin{minipage}{0.4\linewidth}
\center{\includegraphics[width=1\linewidth]{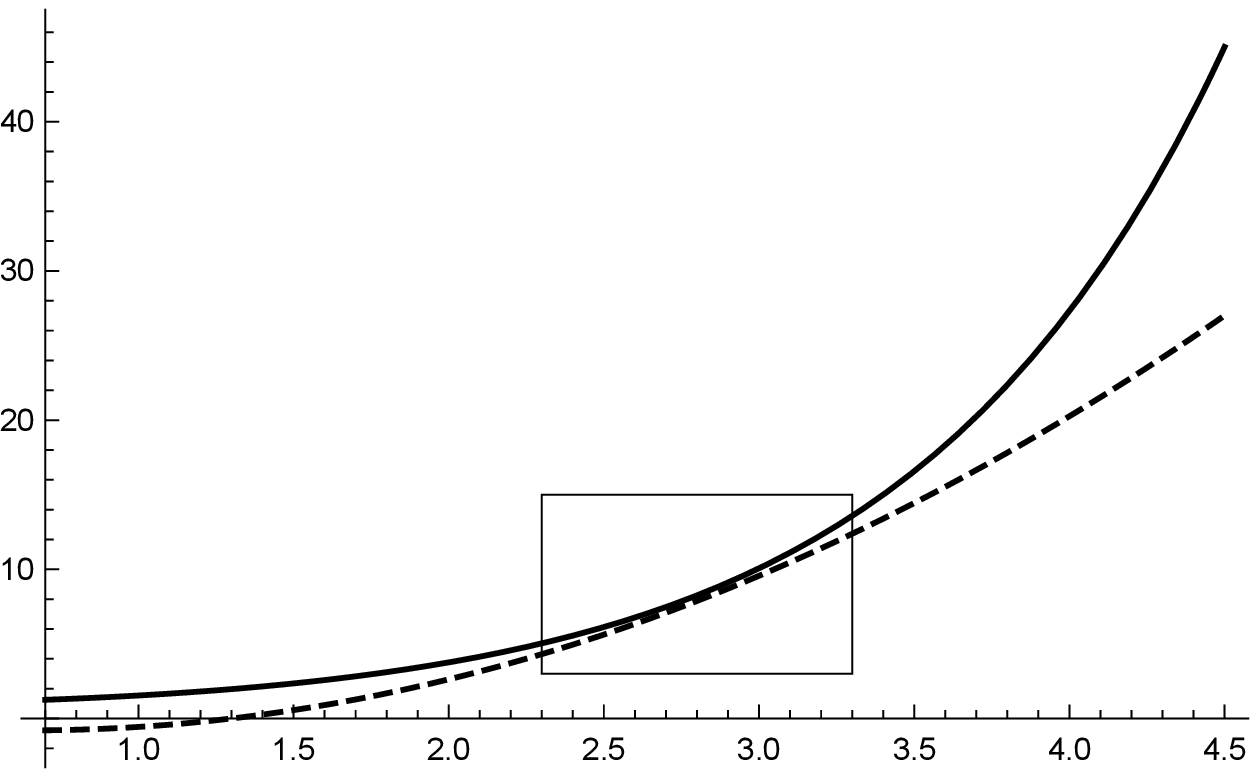}}
\end{minipage}
\hfill
\begin{minipage}{0.4\linewidth}
\center{\includegraphics[width=1\linewidth]{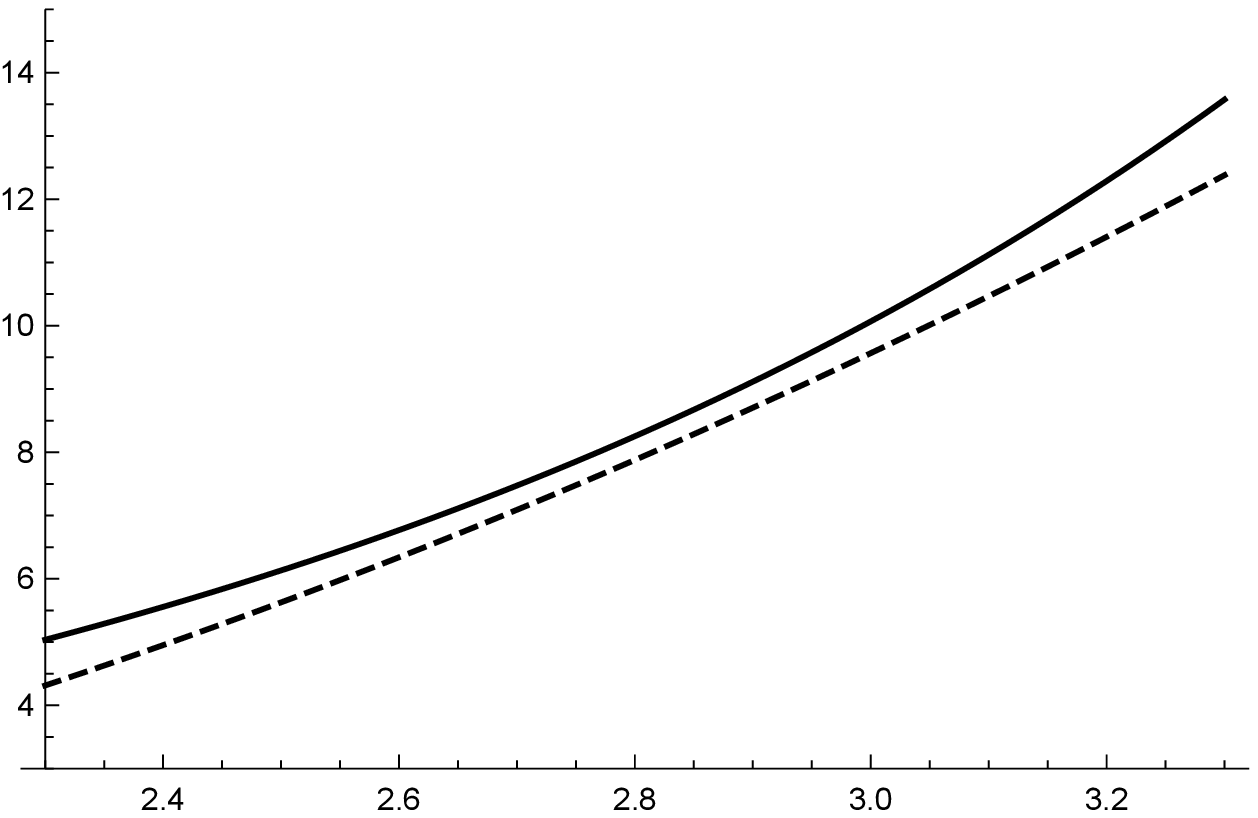}}
\end{minipage}
\caption{Left: the graphs of the functions $\cosh c$ (solid) and $\frac{9}{4T_\infty}c^2-\left(\frac{9}{8T_\infty}+\frac{3}{2}\right)c$ (dashed). Right: the part of the graph inside the box}
\label{fig:efs-num1}
\end{figure}

\textit{Case 2:} $x\in[b-\frac{1}{2},b]$. In this case we estimate the left hand side of~\eqref{eq:efs-fineq} using Lagrange's form of the remainder in Taylor's Theorem. We have (see Fig.~\ref{fig:efs-numder})
$$
f''(x)=
\frac{2x\sinh^2 x-x\cosh^2 x-2\sinh x\cosh x}{\cosh^3 x}<0.1\quad\forall x\in\R.
$$
\begin{figure}[h!]
\begin{minipage}{0.4\linewidth}
\center{\includegraphics[width=1\linewidth]{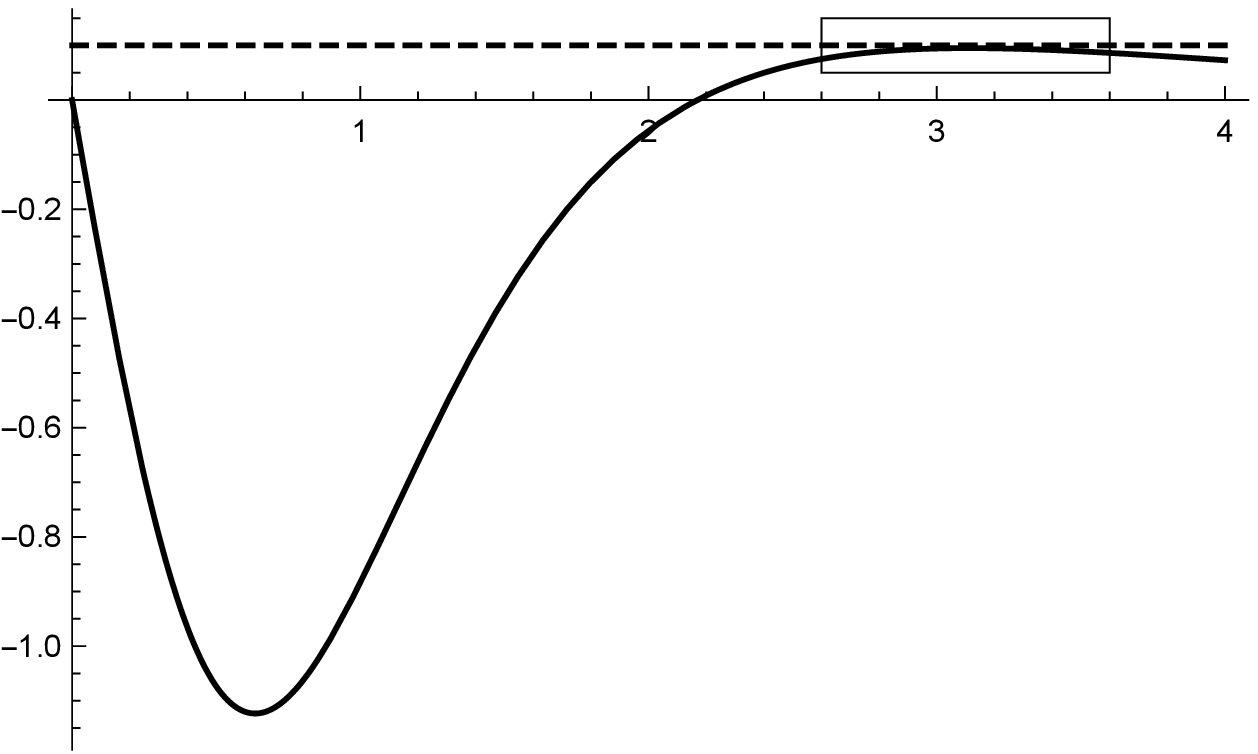}}
\end{minipage}
\hfill
\begin{minipage}{0.4\linewidth}
\center{\includegraphics[width=1\linewidth]{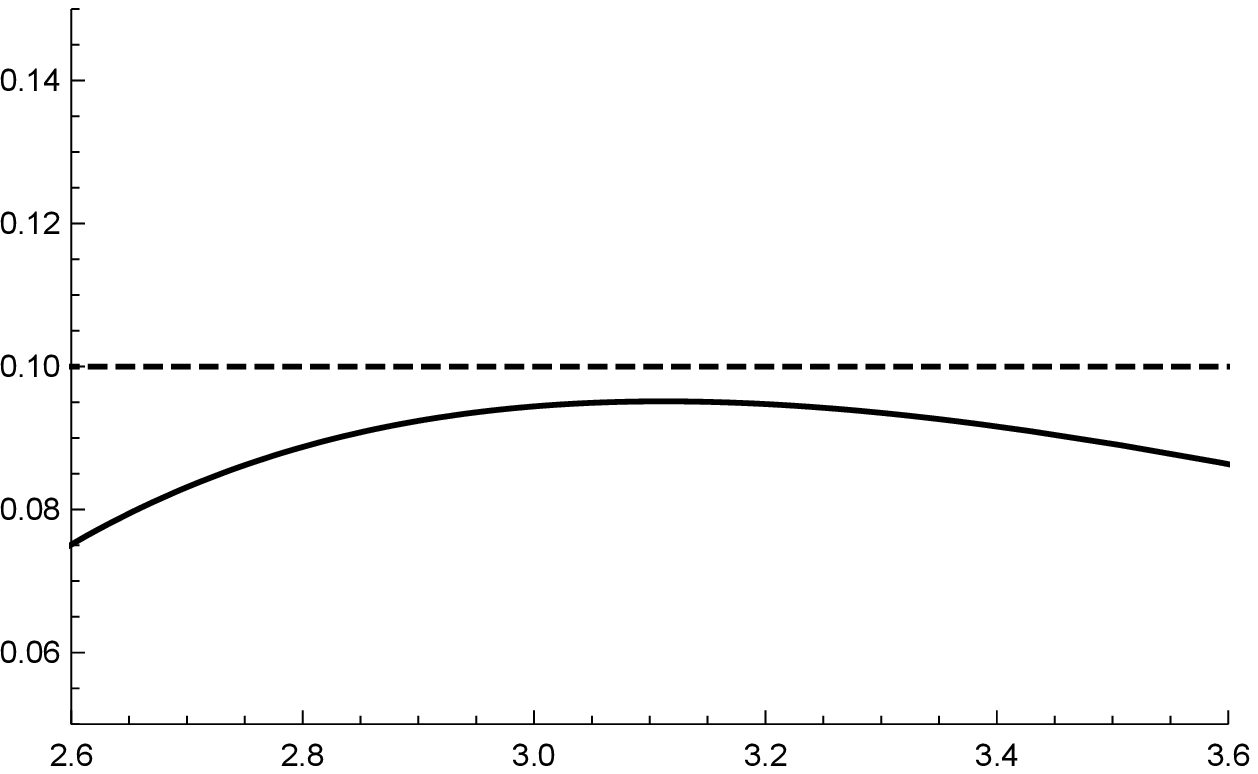}}
\end{minipage}
\caption{Left: the graph of the function $f''(x)$ (solid) and the line $y=0.1$ (dashed). Right: the part of the graph inside the box}
\label{fig:efs-numder}
\end{figure}
This can be checked numerically in a way similar to~\eqref{eq:efs-num1}. It follows that
$$
\frac{d^2}{dx^2}\bigl(f(x)+f(2b-x)\bigr)<0.2\quad\forall x\in\R.
$$
Using the Lagrange's form of the remainder, we obtain
$$
f(x)+f(2b-x)<2f(b)+\frac{0.2}{2}(x-b)^2\leqslant
\frac{2b}{\cosh b}+0.025.
$$
Thus it suffices to show that
$$
\frac{2b}{\cosh b}+0.025<\frac{2b}{3b-2T_\infty}\Leftrightarrow
\cosh b>\frac{2b(3b-2T_\infty)}{1.925b+0.05T_\infty}\quad\text{for}\quad b\geqslant T_\infty.
$$
The latter inequality can be checked numerically similarly to~\eqref{eq:efs-num1} (see Fig.~\ref{fig:efs-num2}).
\begin{figure}[h!]
\begin{minipage}{0.48\linewidth}
\center{\includegraphics[width=1\linewidth]{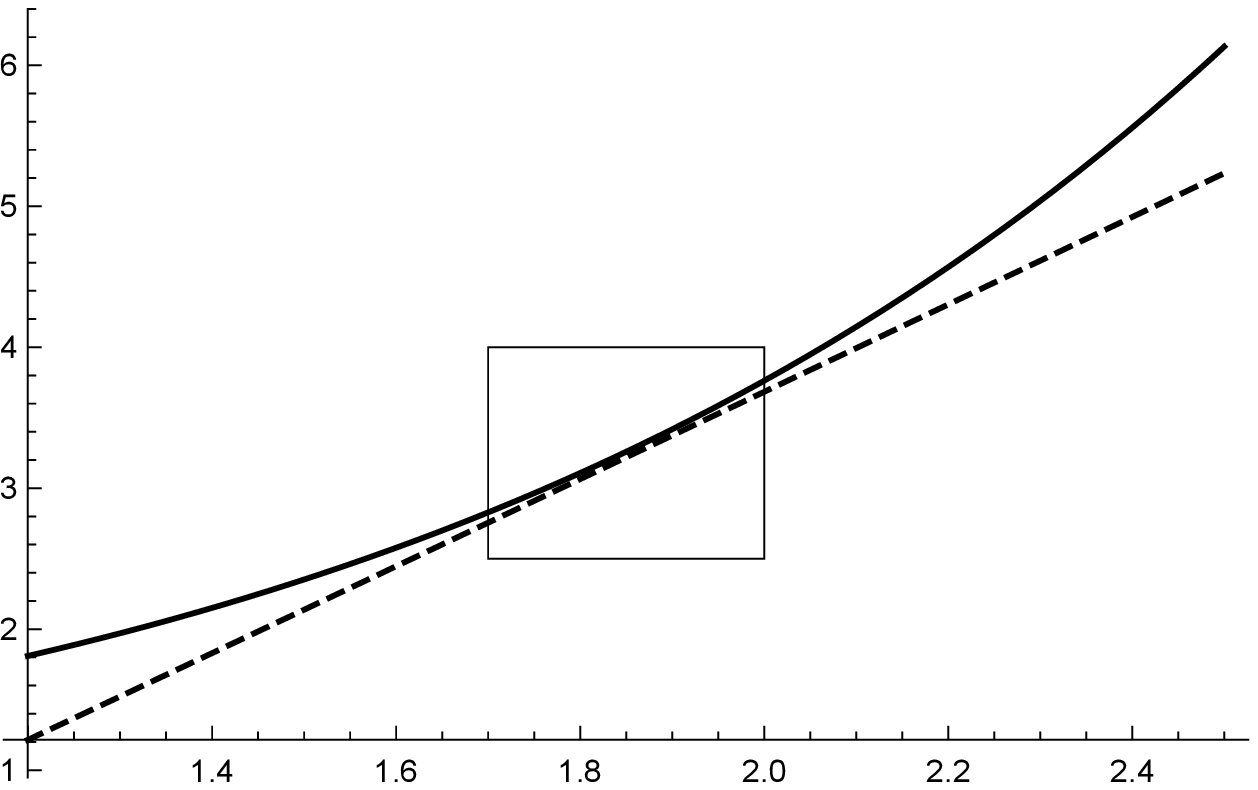}}
\end{minipage}
\hfill
\begin{minipage}{0.48\linewidth}
\center{\includegraphics[width=1\linewidth]{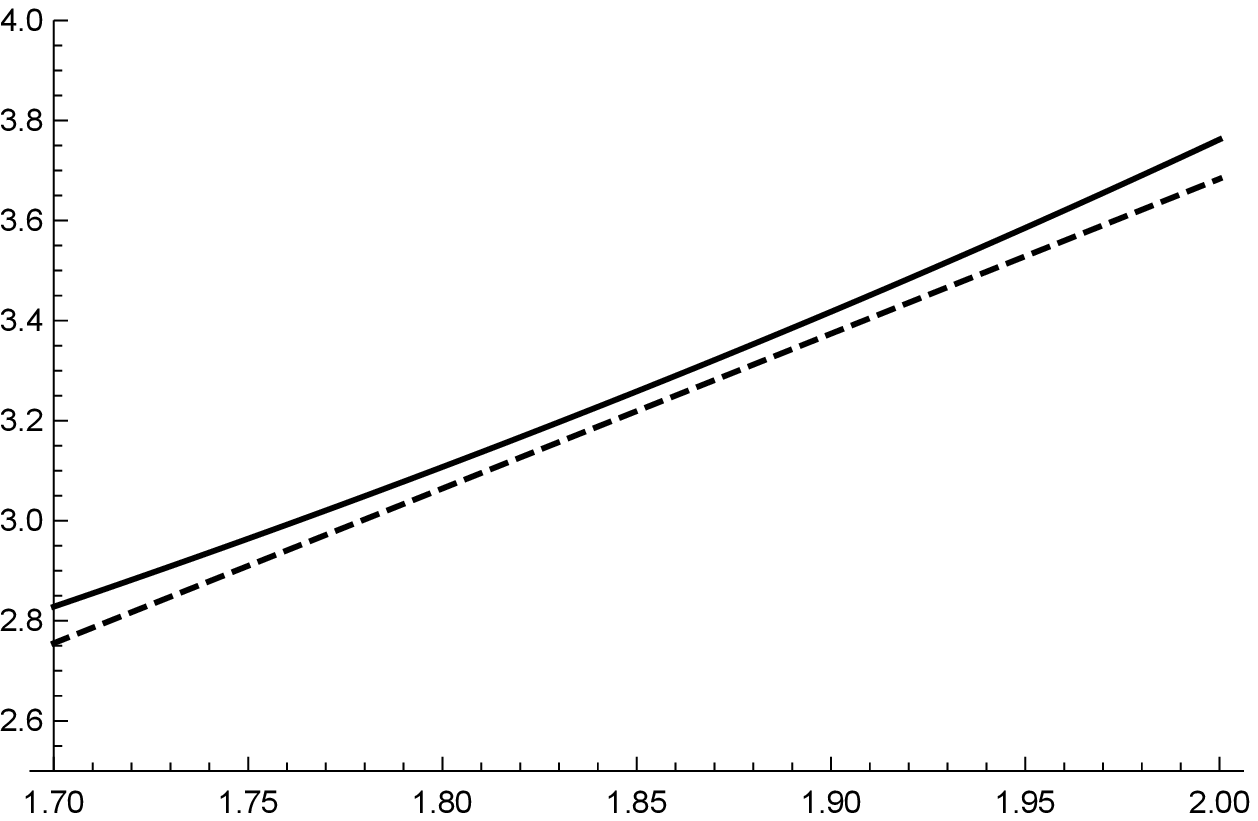}}
\end{minipage}
\caption{Left: the graphs of the functions $\cosh b$ (solid) and $\frac{2b(3b-2T_\infty)}{1.925b+0.05T_\infty}$ (dashed). Right: the part of the graph inside the box}
\label{fig:efs-num2}
\end{figure}

\section{Towards the index of $IFS_{k,l}$}\label{sec:ifs}

In this section we use parametrization~\eqref{eq:intro-param1} with the domain restricted to $[-T_{k,l},T_{k,l}]\times(\R/2\pi\Z)$.

\subsection{Rough upper and lower bounds}

It is possible to obtain a rough upper bound on the value $\Ind(IFS_{k,l})$ in the way similar to~\cite[\S2.2]{morozov2022index}.
Let us recall some definitions. Suppose that $\Sigma$ is an FBMS in $\B^n$.
Then the \emph{Steklov problem} on $\Sigma$ is the following eigenvalue problem
\begin{equation}\label{eq:ifs-steklov}
\begin{cases}
\Delta v=0&\text{ in }\Sigma,\\
\partial_\eta v=\sigma v &\text{ on }\partial\Sigma,
\end{cases}
\end{equation}
where $v\in C^\infty(\Sigma)$.
The real numbers $\sigma$ such that the Steklov problem admits non-trivial solutions are called \emph{Steklov eigenvalues} and form the \emph{Steklov spectrum}.
The corresponding solutions $v$ are called \emph{Steklov eigenfunctions}.
It is well-known that in the case of FBMS the restrictions of coordinate functions to $\Sigma$ are Steklov eigenfunctions with eigenvalue~1.
By definition, the \emph{spectral index} $\Ind_S\Sigma$ of $\Sigma$ is the number of Steklov eigenvalues of $\Sigma$ less than~1 counting with multiplicity.

\begin{proposition}[{\cite[Corollary~6.1]{medvedev2023index}}]\label{pr:ifs-indsub}
Let $\Sigma$ be an FBMS in $\B^n$. Then
$$
\Ind(\Sigma)\leqslant n\Ind_S\Sigma+\dim\mathcal M(\Sigma),
$$
where $\mathcal M(\Sigma)$ is the moduli space of conformal structures on $\Sigma$.
\end{proposition}

\begin{proposition}\label{pr:ifs-indscalc}
We have
$$
\Ind_S(IFS_{k,l})=\begin{cases}
2(k+l-1),&\text{$k$ is odd,}\\
k+l-2,&\text{$k$ is even.}
\end{cases}
$$
\end{proposition}

\begin{proof}
For $\Sigma=\widetilde{IFS}_{k,l}$ the problem~\eqref{eq:ifs-steklov} takes the form
$$
\begin{cases}
v_{tt}+v_{\theta\theta}=0&\text{ in }\Sigma,\\
v_t=\sigma\rho(T_{k,l})v,&\text{ on }\partial\Sigma.
\end{cases}
$$
It is easy to see from the standard separation of variables that this problem has the following eigenvalues and eigenfunctions
\begin{itemize}
\item{the eigenvalue $0=:\lambda(0)$ with eigenfunction $1$;}
\item{the eigenvalue $\frac{1}{T_{k,l}\rho(T_{k,l})}=:\mu(0)$ with eigenfunction $t$;}
\item{the eigenvalue $\frac{m\tanh(m T_{k,l})}{\rho(T_{k,l})}=:\lambda(m)$ with eigenfunctions $\cosh(mt)\cos(m\theta)$ and $\cosh(mt)\sin(m\theta)$ for each $m=1,2,\ldots$;}
\item{the eigenvalue $\frac{m\coth(m T_{k,l})}{\rho(T_{k,l})}=:\mu(m)$ with eigenfunctions $\sinh(mt)\cos(m\theta)$ and $\sinh(mt)\sin(m\theta)$ for each $m=1,2,\ldots$.}
\end{itemize}
Note that $T_{k,l}\rho(T_{k,l})>1$ (this is proved in the same way as the inequality $T_p\rho(T_p)>1$ in \S\ref{sec:efs-auxscal}).
Hence, $\lambda(0)=0$ and $\mu(0)<1$ both contribute~1 to $\Ind_S(\Sigma)$.
Then it is easy to check that both $\lambda(m)$ and $\mu(m)$ are increasing in $m$ and that $\lambda(k)=\mu(l)=1$.
Hence, $\lambda(m)$ contributes~2 to $\Ind_S(\widetilde{IFS}_{k,l})$ for each $m\in [1,k-1]$ and $\mu(m)$ contributes~2  to $\Ind_S(\widetilde{IFS}_{k,l})$ for each $m=[1,l-1]$.
Thus for odd $k$ we obtain
$$
\Ind_S(IFS_{k,l})=\Ind_S(\widetilde{IFS}_{k,l})=2+2(k-1)+2(l-1)=2(k+l-1).
$$
If $k$ is even then only the eigenfunctions satisfying $v(t,\theta)=v(-t,\theta+\pi)$ contribute to $\Ind_S(IFS_{k,l})$.
These are the constant function, the eigenfunctions with eigenvalue $\lambda(m)$ for each even $m\in [2,k-2]$, and the eigenfunctions with eigenvalue $\mu(m)$ for each odd $m\in [1,l-2]$. Thus,
$$
\Ind_S(IFS_{k,l})=1+2\cdot\frac{k-2}{2}+2\cdot\frac{l-1}{2}=k+l-2,
$$
which concludes the proof.
\end{proof}

\begin{proof}[Proof of Theorem~\ref{th:ifs-ub}]
Since the surface $IFS_{k,l}$ is topologically either a cylinder or a M\"obius band, we have $\dim\mathcal M(IFS_{k,l})=1$ independently of $k,l$.
The result now follows directly from Propositions~\ref{pr:ifs-indsub} and~\ref{pr:ifs-indscalc}.
\end{proof}

In the end of this section we prove Corollary~\ref{cor:ifs-lb}. As we mentioned in~\S\ref{sec:intro-ov} this corollary immedeately follows from Propostion~\ref{pr:ifs-split}.

\begin{proof}[Proof of Proposition~\ref{pr:ifs-split}]
Consider a bounded domain $\Omega\subset\Sigma$. Let $V\subset \Gamma(N\Sigma|_\Omega)$ be the maximal subspace in normal fields with support in $\Omega$ on which the second variation $\delta^2A$ of the area functional is negative definite, i.e. $\delta^2A(X)<0$ for any $X\in V$ and $\dim V=\Ind(\Omega)$. We split $\delta^2A(X)$ into two parts in the following way
\begin{gather*}
\delta^2 A(X)=\int_\Sigma (|\nabla^\bot X|^2-\langle\mathcal B(X),X\rangle)\,dA=\\
\left(\int_{I\Sigma} (|\nabla^\bot X|^2-\langle\mathcal B(X),X\rangle)\,dA-\int_{\partial I\Sigma}|X|^2\,dL\right)+\left(\int_{E\Sigma} (|\nabla^\bot X|^2-\langle\mathcal B(X),X\rangle)\,dA+\int_{\partial E\Sigma}|X|^2\,dL\right)=\\= \delta^2A_I(X|_{I\Sigma})+\delta^2A_E(X|_{E\Sigma}),
\end{gather*}
where $\delta^2A_I(X|_{I\Sigma})$ and $\delta^2A_E(X|_{E\Sigma})$ are the second variations of the area functional for $I\Sigma$ and $E\Sigma$ respectively computed on the restrictions of the field $X$ on $I\Sigma$ and $E\Sigma$. This formula is true since $\Sigma=I\Sigma\cup E\Sigma$ and $
\int_{\partial I\Sigma}|X|^2\,dL=\int_{\partial E\Sigma}|X|^2\,dL$. Let $U\subset V$ be a subspace of the vector fields whose restriction on $E\Sigma$ yields $\delta^2A_E(X|_{E\Sigma})$ negative. It is clear that $\dim U \leqslant \Ind(E\Sigma\cap \Omega)$. Since any quadratic form diagonalizes over a finite dimensional vector space, one can assume that $V=U\oplus U'$. Take any vector field $X \in U'$. Then $\delta^2A(X)<0$ and $\delta^2A_E(X|_{E\Sigma})\geqslant 0$. Thus $\delta^2A_I(X|_{I\Sigma})$ is necessarily negative. Hence,
$$
    \Ind(I\Sigma\cap \Omega)\geqslant \dim U' \geqslant \Ind(\Omega)-\Ind(E\Sigma\cap \Omega),
$$
which yields
$$
   \Ind(\Omega) \leqslant \Ind(I\Sigma\cap \Omega)+\Ind(E\Sigma\cap \Omega)\leqslant \Ind(I\Sigma)+\Ind(E\Sigma).
$$
Taking the supremum over all bounded domains $\Omega\subset \Sigma$ from both sides of this inequality completes the proof.
\end{proof}

An application of this proposition to $FS_{k,l}$ immediately implies Corollary~\ref{cor:ifs-lb} since by Theorem~\ref{th:fs-ind}
$$
\Ind(FS_{k,l})=\begin{cases}
2k-1,&\text{$k$ is odd,}\\
k-1,&\text{$k$ is even,}
\end{cases}
$$
and by Conjecture~\ref{con:efs-stab} $\Ind(EFS_{k,l})=0.$

\subsection{Numerical experiments}

As in~\S\ref{sec:fs} it is more convenient to work with $\widetilde{IFS}_{k,l}$ instead of $IFS_{k,l}$.
Recall from~\S\ref{sec:def} that one can define $\Ind(\widetilde{IFS}_{k,l})$ as the number of negative eigenvalues of the following problem
\begin{equation}\label{eq:ifs-eigen}
\begin{cases}
LX=-\lambda X&\text{ on }\widetilde{IFS}_{k,l},\\
\partial_\eta X-X=0&\text{ on }\partial\widetilde{IFS}_{k,l},
\end{cases}
\end{equation}
where $X\in\Gamma(N(\widetilde{IFS}_{k,l}))$.
The following proposition is similar to Propositions~\ref{pr:fs-sep} and~\ref{pr:efs-sep} and is proved in the same manner.

\begin{proposition}
For each $m=0,1,2,\ldots$ consider the following matrix Sturm-Liouville weighted eigenvalue problem with Robin boundary conditions
\begin{equation}\label{eq:ifs-sl}
\begin{cases}
-h''(t)+A(m,t)h(t)=\lambda\rho(t)^2 h(t),\\
h'(T_{k,l})=\rho(T_{k,l})h(T_{k,l}),\quad
h'(-T_{k,l})=-\rho(-T_{k,l})h(-T_{k,l}),
\end{cases}
\end{equation}
where $h(t)=\left(\begin{smallmatrix}h_1(t)\\ h_2(t)\end{smallmatrix}\right)$ is a vector function and $A(m,t)$ as in~\eqref{eq:fs-pot}.
Then the conclusions 1)---3) of Proposition~\ref{pr:fs-sep} hold with~\eqref{eq:ifs-eigen} and~\eqref{eq:ifs-sl} instead of~\eqref{eq:fs-eigen} and~\eqref{eq:fs-sl} respectively and $N(IFS_{k,l})$ instead of $N(FS_{k,l})|_{\Omega_T}$.
\end{proposition}

Denote the $i$-th eigenvalue of the problem~\eqref{eq:ifs-sl} by $\lambda_i(m)$ (where $i\geqslant 1$).
Then it is natural to arrange these eigenvalues in a table that contains $\lambda_i(m)$ at the intersection of the $m$-th row and the $i$-th column (see Table~\ref{tab:ifs-eigenval} for the case $k=7,l=3$).

\begin{table}[h!]
$$
\begin{array}{cccccc}
  & 1 & 2 & 3 & 4 & 5 \\
0 & -5.443 & -2.415 & -1.512 & \mathbf{0.000} & 10.780 \\
1 & -5.393 & -2.441 & -1.387 & 0.055 & 10.850 \\
2 & -5.242 & -2.473 & -1.060 & 0.220 & 11.060 \\
3 & -4.995 & -2.453 & -0.589 & 0.497 & 11.410 \\
4 & -4.654 & -2.358 & \mathbf{0.000} & 0.887 & 11.910 \\
5 & -4.226 & -2.179 & 0.694 & 1.390 & 12.550 \\
6 & -3.717 & -1.914 & 1.488 & 2.010 & 13.350 \\
7 & -3.132 & -1.562 & 2.376 & 2.744 & 14.300 \\
8 & -2.477 & -1.125 & 3.355 & 3.593 & 15.400 \\
9 & -1.757 & -0.603 & 4.419 & 4.553 & 16.670 \\
10 & -0.976 & \mathbf{0.000} & 5.567 & 5.620 & 18.110 \\
11 & -0.136 & 0.684 & 6.788 & 6.793 & 19.710 \\
12 & 0.763 & 1.447 & 8.051 & 8.095 & 21.500 \\
\end{array}
$$
\caption{some eigenvalues $\lambda_i(m)$ computed numerically for $k=7,l=3$. The bold zero eigenvalues come from the rotational Killing fields.}
\label{tab:ifs-eigenval}
\end{table}

Some features of this table can be predicted. First since $\lambda_i(m)$ is nondecreasing in $i$, each row of the table is a nondecreasing sequence.
Secondly considering the rotational Killing fields in $\E^4$ and using Claim~\ref{cl:def-jacfields} we obtain by a direct computation that the vector functions
$$
\frac{1}{\rho(t)}\begin{pmatrix}0\\ k\sinh lt\cosh kt\end{pmatrix},\,
\frac{1}{\rho(t)}\begin{pmatrix}k\sinh lt\cosh lt+l\sinh kt\cosh kt\\ -l\cosh^2 kt+k\sinh^2 lt\end{pmatrix},\,
\frac{1}{\rho(t)}\begin{pmatrix}k\sinh lt\cosh lt+l\sinh kt\cosh kt\\ l\cosh^2 kt+k\sinh^2 lt\end{pmatrix}
$$
solve~\eqref{eq:ifs-sl} with $\lambda=0$ and $m=0,k-l,k+l$ respectively. This means that some columns of the table contain zeros at 0-th, $(k-l)$-th and $(k+l)$-th rows.
The following conjecture is based on the numerical experiments in Wolfram Mathematica.

\begin{conjecture}\label{con:ifs-lb}
The eigenvalue $\lambda_i(m)$ is
\begin{itemize}
\item{negative for $i=1,m\in [0,k+l]$ and $i=2,m\in [0,k+l-1]$ and $i=3,m\in [0,k-l-1]$;}
\item{zero for $(i,m)\in\{(2,k+l),(3,k-l),(4,0)\}$;}
\item{positive for $i=1,m\geqslant k+l$ and $i=2,m\geqslant k-l$ and $i=3,m\geqslant 1$ and $i\geqslant 4$.}
\end{itemize}
In particular,
$$
\Ind(\widetilde{IFS}_{k,l})\geqslant 6k+2l-1.
$$
\end{conjecture}

\begin{remark}
1) Conjecture~\ref{con:ifs-lb} tells nothing about the eigenvalues $\lambda_i(m)$ for $i=1,m>k+l$.
According to our experiments, some of them can be negative.
For example, $\lambda_1(11)<0$ for $k=7,l=3$ (see Table~\ref{tab:ifs-eigenval}).

2) One of the major difficulties in attempts to prove Conjecture~\ref{con:ifs-lb} is that $\lambda_i(m)$ generally is not an increasing function in $m$ in contrast to scalar problems (see~\cite[Corollary~3]{karpukhin2014spectral} and~\cite[Proposition~15]{penskoi2012extremal} for example).
Indeed, the second column of Table~\ref{tab:ifs-eigenval} shows that $\lambda_2(m)$ is not an increasing function in $m$ for $k=7,l=3$.
\end{remark}

\begin{appendices}

\section{Index upper bound for EFBMS}\label{sec:efbms}

In this section we prove the following upper bound for the index of EFBMS in the spirit of~\cite[inequality~(3.1)]{ejiri2008comparison}.

\begin{theorem}\label{th:efbms-ub}
Let $\Sigma$ be an orientable EFBMS in $\E^n\setminus\mathring\B^n$ with finite total curvature, $b$ boundary components, and without branch points. Then
\begin{equation}\label{eq:efbms-ub}
\Ind(\Sigma)\leqslant -\frac{1}{\pi}\TC(\Sigma)+2\gamma-2+b,
\end{equation}
where $\gamma$ is the genus of the Huber-Osserman compactification of $\Sigma$.
\end{theorem}

\begin{remark}
Theorem~\ref{th:efbms-ub} holds true if we replace $\B^n$ by any (non-strictly) convex domain in $\E^n$. Indeed, the proof makes use of that the energy index of the compactified surface is zero which holds for any (non-strictly) convex domain. Therefore Theorem~\ref{th:efbms-ub} can be considered as a complementary result to the result of Ambrosio, Carlotto and Sharp in the paper~\cite{ambrozio2018index}. 
\end{remark}

\begin{remark}\label{rem:efbms-ho}
The Huber-Osserman compactification of an EFBMS is defined in the same way as for minimal surfaces without boundary (see~\S\ref{sec:fs-ub}). More precisely, it is not hard to see that $\Sigma$ is conformally equivalent to a compact Riemann surface with boundary and finite number of punctures. Indeed, glue the boundary components of the surface by disks in a smooth way. We obtain a surface $\Sigma'$ without boundary of finite total curvature. By Huber's theorem~\cite{huber1958subharmonic} $\Sigma'$ is homeomorphic to a closed Riemann surface with finite number of punctures. Then $\Sigma$ is homeomorphic to a compact Riemann surface with boundary and finite number of punctures. The Gauss map of $\Sigma$ is holomorphic (if we introduce a complex coordinate on $\Sigma$). Then the proof of Osserman's theorem~\cite[Theorem 9.1]{osserman2013survey} is valid since all the arguments are local.
\end{remark}

\begin{remark}
1) Note that for the surfaces $EFS_{k,l}$ the inequality~\eqref{eq:efbms-ub} gives
$$
\Ind(EFS_{k,l})\leqslant -\frac{1}{\pi}\TC(EFS_{k,l})-1=2k-1,
$$
which is much weaker than Conjecture~\ref{con:efs-stab}.

2) For catenoidal EFBMS $\mathcal C_\alpha,\alpha\in [0,\frac{\pi}{2})$ in $\E^3\setminus\mathring\B^3$, considered in~\cite[\S4]{mazet2022free}, the inequality~\eqref{eq:efbms-ub} gives $\Ind\mathcal C_\alpha\leqslant 1$.
However, in~\cite[Proposition~5]{mazet2022free} it is shown that $\mathcal C_\alpha$ is stable if $\alpha\in[0,\frac{\pi}{4}]$ and has index~1 otherwise.
Thus, the estimate~\eqref{th:efbms-ub} is not sharp in this case neither. 
\end{remark}

\subsection{Riemann-Roch Theorem for surfaces with boundary}

In this section we recall a version of Riemann-Roch Theorem for surfaces with boundary, as it is stated in~\cite[Theorem~C.1.10]{mcduff2012j} and~\cite[Appendix~A]{lima2017bounds}.

Let $E$ be a complex vector bundle over a compact Riemann surface $\Sigma_c$ with boundary.
Consider $E$ as a $2n$-dimensional real vector bundle equipped with a complex structure $J\colon E\to E$.
A \emph{Hermitian structure} $(-,-)$ on $E$ is a Riemannian metric on $E$ such that $J$ is orthogonal w.r.t. this metric.
Note that a Hermitian structure on $E$ and a volume form on $\Sigma_c$ give rise to a natural Riemannian metric on each bundle $\Lambda^{p,q}\Sigma_c\otimes E$. We use the same notation $(-,-)$ for this metric.
A subbundle $F\subset E|_{\partial\Sigma_c}$ is called \emph{totally real} if $\dim_\R F=n$ and $F_p\cap(JF_p)=\{0\}$ (or, equivalently, $F_p\perp JF_p$) at any point $p\in\partial\Sigma_c$.

Denote by $W^{l,p}(-)$ the space of sections of Sobolev class $W^{l,p}$. Define
$$
W_F^{l,p}(E)=\{X\in W^{l,p}(E)\colon X|_{\partial\Sigma_c}\subset F\},\quad
W_F^{l,p}(E\otimes\Lambda^{0,1}\Sigma_c)=\{\omega\in W^{l,p}(E\otimes\Lambda^{0,1}\Sigma_c)\colon \omega(T\partial\Sigma_c)\subset F\}.
$$

A (complex linear, smooth) \emph{Cauchy-Riemann operator} on $E$ is a $\C$-linear operator $D\colon\Gamma(E)\to\Gamma(E\otimes_\C\Lambda^{0,1}\Sigma_c)$ which satisfies the Leibnitz rule
$$
D(\varphi X)=\varphi(DX)+(\bar\partial\varphi)\otimes X,\quad
X\in\Gamma(E),\,\varphi\in C^\infty(\Sigma_c,\C).
$$
Let $l$ be a positive integer and $p>1$ such that $lp>2$. A \emph{real linear Cauchy-Riemann operator} of class $W^{l-1,p}$ on $E$ is an operator of the form
$$
D=D_0+\alpha,
$$
where $\alpha\in W^{l-1,p}(\mathrm{End}_\R(E)\otimes\Lambda^{0,1}\Sigma_c)$ and $E$ is a complex linear Cauchy-Riemann operator on $E$.

\begin{theorem}[Riemann-Roch for surfaces with boundary]\label{th:efbms-rr}
Let $E\to\Sigma$ be a complex vector bundle of complex dimension $n$ over a compact Riemann surface with boundary.
Supply $E$ and $\Sigma_c$ with a Hermitian structure $\langle-,-\rangle$ and a volume form $dA$ respectively.
Let $F\subset E|_{\partial\Sigma_c}$ be a totally real subbundle.
Let $D$ be a real linear Cauchy-Riemann operator on $E$ of class $W^{l-1,p}$, where $l$ is a positive integer and $p>1$ such that $lp>2$.
Then the following holds for every integer $k\in\{1,\ldots,l\}$ and every real number $q>1$ such that $k-\frac{2}{q}\leqslant l-\frac{2}{p}$.
\begin{enumerate}[label=(\roman*)]
\item{The operators
$$
D_F\colon W_F^{k,q}(E)\to W^{k-1,q}(E\otimes\Lambda^{0,1}\Sigma_c),\quad
D_F^*\colon W_F^{k,q}(E\otimes\Lambda^{0,1}\Sigma_c)\to W^{k-1,q}(E)
$$
are Fredholm. Moreover, their kernels are independent of $k$ and $q$, and we have
$$
\omega\in\im D_F\quad\Longleftrightarrow\quad
\int_{\Sigma_c} (\omega,\omega_0)\,dA=0\quad\forall\omega_0\in\ker D_F^*
$$
and
$$
X\in\im D_F^*\quad\Longleftrightarrow\quad
\int_{\Sigma_c} (X,X_0)\,dA=0\quad\forall X_0\in\ker D_F.
$$
}
\item{The real Fredholm index of $D_F$ is given by
\begin{equation}\label{eq:efbms-rrind}
\mathrm{index}(D_F)=n\chi(\Sigma_c)+\mu(E,F),
\end{equation}
where $\chi(\Sigma_c)$ is the Euler characteristic of $\Sigma$ and $\mu(E,F)$ is the boundary Maslov index.
}
\item{If $n=1$, then
\begin{gather}
\mu(E,F)<0\Longrightarrow D_F\text{ is injective},\label{eq:efbms-rrinj}\\
\mu(E,F)+2\chi(\Sigma_c)>0\Longrightarrow D_F\text{ is surjective}.\notag
\end{gather}
}
\end{enumerate}
\end{theorem}
The following proposition is useful for calculation of the Maslov index.

\begin{proposition}\label{pr:efbms-maslov}
Let $E\to\Sigma_c$ be a complex vector bundle over a compact Riemann surface with boundary and $F\subset E|_{\partial\Sigma_c}$ a totally real subbundle.
Let $\Sigma_c^1,\Sigma_c^2$ be two copies of $\Sigma$ and $\widetilde\Sigma=\Sigma_c^1\cup\Sigma_c^2$ the double of $\Sigma$.
Suppose that $\widetilde E\to\widetilde\Sigma_c$ is a complex vector bundle such that the identifying maps $\Sigma\to\Sigma_c^1$ and $\Sigma\to\Sigma_c^2$ extend to vector bundle isomorphisms $i_1\colon E\to\widetilde E|_{\Sigma_c^1}$ and $i_2\colon E\to\widetilde E|_{\Sigma_c^2}$ such that $i_1(F)=i_2(F)$. Then
$$
\mu(E,F)=c_1(\widetilde E).
$$
\end{proposition}

\begin{proof}
This follows directly from~\cite[Theorems~C.3.5 and~C.3.10]{mcduff2012j}, cf.~\cite[Example in Appendix~A]{lima2017bounds}.
\end{proof}

\subsection{Proof of Theorem~\ref{th:efbms-ub}}

The proof goes along the lines of~\cite[proof of Theorem~1]{lima2017bounds} (see also~\cite[proof of Theorem~3.2]{ejiri2008comparison}). 
Let $\Sigma_c$ be the Huber-Osserman compactification of $\Sigma$ (see Remark~\ref{rem:efbms-ho}) and define the bundles $\tau,\nu$ as in~\S\ref{sec:fs-ub}.
Then, again, the quadratic forms $\delta^2 A$ and $\delta^2 E$ extend to $\Gamma(\nu)$ and $\Gamma(\Sigma_c\times\R^n)$ and~\eqref{eq:fs-sigmacind} holds.
Let $z=x+iy$ be a local conformal coordinate on $\Sigma$ such that at any point $p\in\partial\Sigma$ we have $\partial_x\in T_p(\partial\Sigma)$ and $\partial_y\perp T_p(\partial\Sigma)$.
Let $F\subset\tau^{0,1}|_{\partial\Sigma}$ and $F'\subset\tau^{1,0}|_{\partial\Sigma}$ be the subbundles whose sections locally are of the form $f\,\partial_{\bar z}$ and $if\,\partial_z$ where $f$ is a purely real function on $\partial\Sigma$.
Then $F$ and $F'$ are totally real subbundles. Define
\begin{align*}
h_F^0(\tau^{0,1}\otimes\Lambda^{1,0}\Sigma_c)&=\dim_\R\{\omega\in\Gamma(\tau^{0,1}\otimes\Lambda^{1,0}\Sigma_c)\colon \omega\text{ is holomorphic and }\omega(T\partial\Sigma)\subset F\},\\
h_{F'}^0(\tau^{1,0})&=\dim_\R\{X\in\Gamma(\tau^{1,0})\colon X\text{ is holomorphic and } X|_{\partial\Sigma}\subset F'\}.
\end{align*}

In the same manner as in~\cite[proof of Theorem~1]{lima2017bounds} and~\cite[proof of Theorem~3.2]{ejiri2008comparison} one can show that
\begin{equation}\label{eq:efbms-ineqpre}
\Ind(\Sigma_c)\leqslant\Ind_E(\Sigma_c)+h_F^0(\tau^{0,1}\otimes\Lambda^{1,0}\Sigma).
\end{equation}
Apply Theorem~\ref{th:efbms-rr} to the bundle $E=\tau^{1,0}$ and the operator
$$
D_{F'}=\bar\partial_E\colon W_{F'}^{k,q}(\tau^{1,0})\to W^{k-1,q}(\tau^{1,0}\otimes\Lambda^{0,1}\Sigma_c).
$$
It is well-known that the adjoint operator is given by
$$
D_{F'}^*=\bar\partial_E^*=-\bar *_E^{-1}\,\bar\partial_{E^*\otimes\Lambda^{1,0}\Sigma_c}\,\bar *_E,
$$
where $\bar *_E\colon E\otimes\Lambda^{p,q}\Sigma_c\to E^*\otimes\Lambda^{1-p,1-q}\Sigma_c$ is the Hodge star operator.
Since $\bar *_E$ is an isomorphism, it is easy to see that
$$
\dim_\R\ker D_{F'}=h_{F'}^0(\tau^{1,0}),\quad
\dim_\R\ker D_{F'}^*=h_F^0(\tau^{0,1}\otimes\Lambda^{1,0}\Sigma_c).
$$
It follows from Proposition~\ref{pr:efbms-maslov} that
$$
\mu(E,F')=\frac{1}{\pi}\TC(\Sigma_c)<0.
$$
Hence, from~\eqref{eq:efbms-rrind} we get
$$
h_{F'}^0(\tau^{1,0})-h_F^0(\tau^{0,1}\otimes\Lambda^{1,0}\Sigma_c)=\chi(\Sigma_c)+\frac{1}{\pi}\TC(\Sigma_c),
$$
and from~\eqref{eq:efbms-rrinj} we get $h_{F'}^0(\tau^{1,0})=0$. Thus,
$$
h_F^0(\tau^{0,1}\otimes\Lambda^{1,0}\Sigma_c)=
-\chi(\Sigma_c)-\frac{1}{\pi}\TC(\Sigma_c).
$$
Combining this with~$\Ind_E(\Sigma_c)=\Ind_E(\Sigma)=0$ since $\partial \B^n$ is convex, $\Ind(\Sigma_c)=\Ind(\Sigma)$ as in the proof of Proposition~\ref{pr:fs-ub},~\eqref{eq:efbms-ineqpre}, and using that $\chi(\Sigma_c)=2-2\gamma-b$, we obtain the desired result.\qed

\section{Index estimates for non-orientable surfaces}\label{sec:nonor}

\subsection{Index upper bound for non-orientable surfaces in $\E^n$}

In this section we prove an analog of~\cite[Theorem 3.2]{ejiri2008comparison} for non-orientable surfaces in $\E^n$. 

\begin{theorem}\label{EMM}
Let $\Sigma$ be a (possibly branched) immersed complete non-orientable minimal surface of genus $\gamma$ of finite total curvature in $\mathbb E^n$. Then
$$
\Ind(\Sigma)+\Nul(\Sigma)\leqslant -\frac{1}{\pi}\TC(\Sigma)+\gamma+n-1.
$$
\end{theorem}

\begin{proof}
Throughout the proof we use the notation of~\S\ref{sec:fs-ub}. Pass to the orientable two-sheeted cover $\widetilde\Sigma$ of $\Sigma$. Let $\widetilde\Sigma_c$ be the Huber-Osserman compactification of $\widetilde\Sigma$ and define the bundles $\tau,\nu$ over $\widetilde\Sigma$ as in~\S\ref{sec:fs-ub}. If $\widetilde\Sigma$ has branch points, then the bundle $\tau$ is twisted at the branch points on amount equal to the order of branching. Then, again, the quadratic forms $\delta^2 A$ and $\delta^2 E$ extend to $\Gamma(\nu)$ and $\Gamma(\widetilde\Sigma_c\times\R^n)$ and~\eqref{eq:fs-sigmacind} holds. Let $\iota$ be the involution changing the orientation of $\widetilde\Sigma$. Consider the space of $\iota-$invariant holomorphic sections of the bundle $\tau^{0,1}\otimes\Lambda^{1,0}\widetilde\Sigma$ that we denote by $H^0_\iota(\tau^{0,1}\otimes\Lambda^{1,0}\widetilde\Sigma_c)$. As in the proof of \cite[Theorem 1.1]{ejiri2008comparison} one can show that
$$
\Ind(\Sigma)+\Nul(\Sigma)=\Ind(\Sigma_c)+\Nul(\Sigma_c) \leqslant \Ind_E(\Sigma_c)+\Nul_E(\Sigma_c)+h^0_\iota,
$$
where $h^0_\iota=\dim_{\mathbb R} H^0_\iota(\tau^{0,1}\otimes\Lambda^{1,0}\widetilde\Sigma_c)$. One also obviously gets that
$$
\Ind_E(\Sigma_c)\leqslant \Ind_E(\widetilde\Sigma_c)=\Ind_E(\widetilde\Sigma)=0.
$$
Hence $\Ind_E(\Sigma)=0$. One also has that $\Nul_E(\Sigma)\leqslant \Nul_E(\widetilde\Sigma)=n$ (see \S\ref{sec:def}). Hence,
\begin{gather}\label{main}
\Ind(\Sigma)+\Nul(\Sigma) \leqslant h^0_\iota+n.
\end{gather}
Now the question is to estimate $h^0_\iota$. Notice that any section $s\in H^0(\tau^{0,1}\otimes\Lambda^{1,0}\widetilde\Sigma_c)$ can be decomposed into the sum of the $\iota-$invariant and the $\iota-$antiinvariant parts in the following way
$$
s(x)=\frac{s(x)+s(\iota(x))}{2}+\frac{s(x)-s(\iota(x))}{2},~x\in \widetilde\Sigma_c.
$$
Let us denote the space of $\iota-$antiinvariant sections of $H^0(\tau^{0,1}\otimes\Lambda^{1,0}\widetilde\Sigma_c)$ as $H^0_a(\tau^{0,1}\otimes\Lambda^{1,0}\widetilde\Sigma_c)$. Then one has
$$
H^0(\tau^{0,1}\otimes\Lambda^{1,0}\widetilde\Sigma_c)=H^0_\iota(\tau^{0,1}\otimes\Lambda^{1,0}\widetilde\Sigma_c)\oplus H^0_a(\tau^{0,1}\otimes\Lambda^{1,0}\widetilde\Sigma_c).
$$
Let $\dim_{\mathbb R} H^0(\tau^{0,1}\otimes\Lambda^{1,0}\widetilde\Sigma_c)=2h^0, \dim_{\mathbb R}H^0_a(\tau^{0,1}\otimes\Lambda^{1,0}\widetilde\Sigma_c)=h^0_a$. Then we have
$$
2h^0=h^0_\iota+h^0_a.
$$
By the Riemann-Roch Theorem one has
$$
\dim_{\mathbb C}H^0(\tau^{1,0})-\frac{1}{2}(h^0_\iota+h^0_a)=\deg(\tau^{1,0})+1-\gamma.
$$
Obviously, $\dim_{\mathbb C}H^0(T^{1,0}\widetilde\Sigma)=0$ since $c_1(T^{1,0}\widetilde\Sigma)=\frac{1}{2\pi}\TC(\widetilde\Sigma)<0$.  Then
\begin{gather*}
\frac{1}{2}(h^0_\iota+h^0_a)=-\frac{1}{2\pi}\TC(\widetilde\Sigma)+\gamma-1.
\end{gather*}
Therefore, we have 
\begin{gather}\label{sum}
h^0_\iota+h^0_a=-\frac{1}{\pi}\TC(\widetilde\Sigma)+2\gamma-2.
\end{gather}
In order to get the second equality on $h^0_\iota$ and $h^0_a$ we use the Lefschetz Fixed Point Theorem. Since the involution $\iota\colon \widetilde\Sigma_c \to \widetilde\Sigma_c$ has no fixed points then the Lefschetz Fixed Point Theorem implies
\begin{gather}\label{Lefschetz}
0=\tr \iota^*|_{H^0(\widetilde\Sigma_c,\mathcal T^{0,1}\otimes\Omega^{1,0}\widetilde\Sigma_c)}-\tr \iota^*|_{H^1(\widetilde\Sigma_c,\mathcal T^{0,1}\otimes\Omega^{1,0}\widetilde\Sigma_c)},
\end{gather}
where $\iota^*|_{H^q(\widetilde\Sigma_c,\mathcal T^{0,1}\otimes\Omega^{1,0}\widetilde\Sigma_c)}$ is the induced action of $\iota$ on the sheaf cohomology group $H^q(\widetilde\Sigma_c,\mathcal T^{0,1}\otimes\Omega^{1,0}\widetilde\Sigma_c)$. Choosing the basis in $H^0(\widetilde\Sigma_c,\mathcal T^{0,1}\otimes\Omega^{1,0}\widetilde\Sigma_c)=H^0(\tau^{0,1}\otimes\Lambda^{1,0}\widetilde\Sigma_c)$ consisting of $h^0_\iota$ $\iota-$invariant sections and $h^0_a$ $\iota-$antiinvariant sections one may easily see that 
$$
\tr \iota^*|_{H^0(\widetilde\Sigma_c,\mathcal T^{0,1}\otimes\Omega^{1,0}\widetilde\Sigma_c)}=h^0_\iota-h^0_a.
$$
Indeed, the $\iota-$invariant sections are those who have eigenvalue $+1$ with respect to the action $\iota^*$ and the $\iota-$antiinvariant sections are those who have eigenvalue $-1$ with respect to the action $\iota^*$. Further, by the Serre duality one has 
$$
H^1(\widetilde\Sigma_c,\mathcal T^{0,1}\otimes\Omega^{1,0}\widetilde\Sigma_c)\cong H^0(\widetilde\Sigma_c, \mathcal T^{0,1})^*. 
$$
Further we see that $H^0(\widetilde\Sigma_c,\mathcal T^{0,1})=H^0(\tau^{0,1})=0$ as we have already discussed. Hence, $H^1(\widetilde\Sigma_c,\mathcal T^{0,1}\otimes\Omega^{1,0}\widetilde\Sigma_c)=0$ and 
$$
\tr \iota^*|_{H^1(\widetilde\Sigma_c,\mathcal T^{0,1}\otimes\Omega^{1,0}\widetilde\Sigma_c)}=0.
$$
Then \eqref{Lefschetz} takes the form
\begin{gather}\label{Lefschetz1}
h^0_\iota-h^0_a=0.
\end{gather}
Equalities \eqref{Lefschetz1} and \eqref{sum} then imply
$$
h^0_\iota=-\frac{1}{2\pi}\TC(\widetilde\Sigma)+\gamma-1.
$$
Substituting the latter into \eqref{main} we get
$$
\Ind(\Sigma)+\Nul(\Sigma) \leqslant -\frac{1}{2\pi}\TC(\widetilde\Sigma)+\gamma-1+n=-\frac{1}{\pi}\TC(\Sigma)+\gamma+n-1.
$$
\end{proof}

\subsection{Ejiri-Micallef type inequalities for non-orientable surfaces} In this section we state a result analogous to the Ejiri-Micallef inequalities~\cite[Theorem 1.1]{ejiri2008comparison} for non-orientable surfaces. This result was used in~\cite[Section 6]{karpukhin2021stability} in the case when $\Sigma$ is the Klein bottle $\tilde\tau_{3,1}$ (the bipolar Lawson surface) in the round sphere $\mathbb S^4$.

\begin{theorem}
Let $\Sigma$ be a (possibly branched) immersed closed  non-orientable minimal surface of genus $\gamma$ in a Riemannian manifold $(M,g)$. Then
$$
\Ind_E(\Sigma) \leqslant \Ind(\Sigma) \leqslant \Ind_E(\Sigma)+r, 
$$
where, if $b$ is the number of branch points of the immersion counted with multiplicity, then
$$
r=\begin{cases} 3\gamma-3-b, &\text{if $b\leqslant \gamma-2$},\\
                          2\gamma-1-b, &\text{if $\gamma-1\leqslant b \leqslant 2\gamma-2$},\\
                          0, &\text{if $b \geqslant 2\gamma-1$}. 
                          \end{cases}
$$
Particularly, if $\gamma=0$, i.e. $\Sigma$ is $\mathbb RP^2$ then $r=0$. If $\gamma=1$, i.e. $\Sigma$ is $\mathbb{KL}$ then $r=1$ if $b=0$ and $r=0$ if $b>0$.
\end{theorem}

\begin{proof}
The proof is analogous to the proof of Theorem~\ref{EMM}. First we show that 
$$
\Ind(\Sigma) \leqslant \Ind_E\Sigma+h^0_\iota,
$$
where $h^0_\iota=\dim_{\mathbb R} H^0_\iota(\tau^{0,1}\otimes\Lambda^{1,0}\widetilde\Sigma)$ is the dimension of the space of the $\iota-$invariant sections of the bundle $\tau^{1,0}\otimes\Lambda^{1,0}\widetilde\Sigma$, $\tau$ is the tangent bundle over $\widetilde\Sigma$ twisted at the branch points, and $\widetilde\Sigma$ is the orientable cover of $\Sigma$. The lower bound $\Ind_E(\Sigma) \leqslant \Ind(\Sigma)$ is trivial. Next, using the Riemann-Roch Theorem, the Lefschetz Fixed Point Theorem, and the Serre duality we show that $h^0_\iota=h^0$, where $2h^0=\dim_{\mathbb R} H^0(\Sigma,\tau^{0,1}\otimes\Lambda^{1,0}\widetilde\Sigma)$. As it was computed in the proof of~\cite[Theorem 1.1]{ejiri2008comparison} by the Riemann-Roch Theorem
$$
h^0=3\gamma-3-\tilde b+\dim_{\mathbb C} H^0(\tau^{0,1}),
$$
where $\tilde b=2b$ is the number of branched points on $\widetilde\Sigma$. If $\tilde b \leqslant 2\gamma-3$ then $c_1(\tau^{0,1})<0$ hence $\dim_{\mathbb C} H^0(\tau^{0,1})=0$ and $h^0=3\gamma-3-\tilde b$. If $\tilde b\geqslant 4\gamma-3$ then $c_1(\tau^{0,1}\otimes\Lambda^{1,0}\widetilde\Sigma)<0$ hence $h^0=0$. If $2\gamma-2 \leqslant \tilde b\leqslant 4\gamma-4$ then $0\leqslant c_1(\tau^{0,1}\otimes\Lambda^{1,0}\widetilde\Sigma)\leqslant 2\gamma-2$ and by Clifford's theorem $h^0\leqslant \left[\frac{4\gamma-2-\tilde b}{2}\right]$.
\end{proof}

\end{appendices}

\bibliographystyle{alpha}
\bibliography{mybib}
\end{document}